\documentclass[11pt]{article}
\usepackage{amsfonts,amsmath,amsthm,amssymb,enumerate,color}

\allowdisplaybreaks 

\usepackage[T1]{fontenc}
\usepackage[utf8]{inputenc}
\usepackage{lmodern}

\usepackage[english]{babel}

\usepackage{bm}			
\DeclareMathAlphabet{\mathbbold}{U}{bbold}{m}{n}	

\usepackage[hypertexnames=false,hyperindex,breaklinks]{hyperref}
\hypersetup{
    colorlinks,
    linkcolor={red!50!black},
    citecolor={blue!50!black},
    urlcolor={blue!80!black},
}

\usepackage{mathtools}
\mathtoolsset{showonlyrefs,showmanualtags}	

\numberwithin{equation}{section}
\newcommand{\beq}{\begin{equation}}
\newcommand{\eeq}{\end{equation}}
\newcommand{\bea}{\begin{eqnarray}}
\newcommand{\eea}{\end{eqnarray}}
\newcommand{\beas}{\begin{eqnarray*}}
\newcommand{\eeas}{\end{eqnarray*}}

\newtheorem{theorem}{Theorem}[section]

\newtheorem{definition}[theorem]{Definition}
\newtheorem{proposition}[theorem]{Proposition}

\newtheorem{corollary}[theorem]{Corollary}
\newtheorem{lemma}[theorem]{Lemma}
\newtheorem{remark}[theorem]{Remark}
\newtheorem{example}[theorem]{Example}
\newtheorem{examples}[theorem]{Examples}
\newtheorem{foo}[theorem]{Remarks}

\newcommand{\V}{\mathcal V}

\newcommand{\M}{\mathbb M}

\newcommand{\R}{\mathbb R}

\newcommand{\ve}{\varepsilon}

\newcommand{\Ric}{\mathrm{Ric}}
\newcommand{\ch}{\mathcal H}

\newcommand{\Tor}{\mathrm{Tor}}
\newcommand{\hor}{\mathcal{H}}
\newcommand{\ver}{\mathcal{V}}
\newcommand{\spn}{\mathrm{span}}

\DeclareMathOperator{\Cut}{\mathbf{Cut}}
\DeclareMathOperator{\diam}{\mathrm{diam}}
\DeclareMathOperator{\rank}{\mathrm{rank}}

\newcommand{\sas}{\mathrm{Sas}}
\newcommand{\rie}{\mathrm{Riem}}
\newcommand{\hty}{\mathrm{Htype}}
\newcommand{\End}{\mathrm{End}}

\newcommand{\Addresses}{{
  \bigskip
  \footnotesize

\noindent  F.~Baudoin, \textsc{Department of Mathematics, Aarhus University
Ny Munkegade 118
8000 Aarhus C}\par\nopagebreak
\noindent  \textit{E-mail address}: \texttt{fbaudoin@math.au.dk}

  \medskip

\noindent  E.~Grong, \textsc{University of Bergen, Department of Mathematics,  P.O. Box 7803, 5020 Bergen, Norway}\par\nopagebreak
\noindent  \textit{E-mail address}, E.~Grong: \texttt{erlend.grong@uib.no}

  \medskip

\noindent L. ~Rizzi, \textsc{SISSA, Trieste, Italy}\par\nopagebreak
\noindent\textsc{Univ. Grenoble Alpes, IF, F-38000 Grenoble, France}\par\nopagebreak
\noindent\textsc{CNRS, IF, F-38000 Grenoble, France}\par\nopagebreak
\noindent  \textit{E-mail address}: \texttt{lrizzi@sissa.it}

 \medskip

\noindent  S. ~Vega-Molino, \textsc{Sjømilitær Teknologi, FHS Sjøkrigsskolen
Sjøkrigsskoleveien 32
5165 Laksevåg, Norway}\par\nopagebreak
\noindent  \textit{E-mail address}, S.~Vega-Molino: \texttt{svegamolino@mil.no} 

\medskip

\noindent Data sharing not applicable to this article as no datasets were generated or analysed during the current study.
}}

\usepackage{todonotes}
\title{Comparison theorems on H-type sub-Riemannian manifolds}

\author{Fabrice Baudoin, Erlend Grong, Luca Rizzi, Sylvie Vega-Molino}

\begin{document}

\maketitle

\begin{abstract}
On H-type sub-Riemannian manifolds we establish sub-Hessian and sub-Laplacian comparison theorems which are uniform for a family of approximating Riemannian metrics converging to the sub-Riemannian one. We also prove a sharp sub-Riemannian Bonnet-Myers theorem that extends to this general setting results previously proved on contact and quaternionic contact manifolds.
\end{abstract}

\setcounter{tocdepth}{2}
\tableofcontents

\section{Introduction}

Let $(\M,g)$ be a Riemannian manifold, equipped with an orthogonal splitting $T\M = \hor \oplus \ver$ of the tangent bundle. The sub-bundles $\hor$ and $\ver$ are called the \emph{horizontal} and \emph{vertical bundles}, respectively. The \emph{canonical variation} of the metric, introduced in \cite[Ch.\ 9]{Besse} in the context of submersions, is the one-parameter family of Riemannian metrics defined by
\begin{equation}\label{eq:canvarintro}
g_\ve = g_\hor \oplus \frac{1}{\ve} g_\ver, \qquad \forall \ve >0,
\end{equation}
where $g_\hor$ and $g_\ver$ denote the restriction of $g$ to the corresponding sub-bundle.

Assuming that $\hor$ is completely non-holonomic, the sequence of Riemannian distances $d_\ve$ converges, as $\ve \to 0$, to the Carnot-Carath\'eodory metric $d_0$ of the sub-Riemannian manifold $(\M,\hor,g_\hor)$. Conversely, the metric structure of any sub-Riemannian manifold can be obtained as a limit of Riemannian metrics in this way. We emphasize, however, that that the choice of the Riemannian extension is not unique.

Even though the convergence $d_\ve \to d_0$ holds in the $C^\infty$ compact-open topology outside of the cut locus, it is not hard to check that the Riemannian curvature of $g_\ve$ is unbounded below as $\ve \to 0$, provided that $\hor$ is completely non-holonomic (see e.g.\ \cite[Ch.\ 2.4]{CDPT-Introduction} for the case of the Heisenberg group). For this reason, several classical geometrical estimates based on the theory of Riemannian curvature lower bounds do not bring, at least directly, any insight for the sub-Riemannian limit structure.

To better illustrate this fact, let us consider as an example the classical Laplacian comparison theorem. This is a fundamental result in Riemannian geometry, directly related to volume comparison theorems and a special case of the Rauch comparison theorem. It states that on a Riemannian manifold $(\M,g)$ with Ricci curvature not smaller than $(n-1)\kappa$, the Laplacian of the distance $r(\cdot)=d(x,\cdot)$ from any $x\in \M$ satisfies
\begin{equation}\label{eq:lapcompintro}
\Delta r \leq \begin{cases}
(n-1)\sqrt{\kappa}\cot(\sqrt{\kappa}r) & \kappa >0, \\
\tfrac{n-1}{r} & \kappa =0,\\
(n-1)\sqrt{|\kappa|}\coth(\sqrt{|\kappa|}r) & \kappa < 0,
\end{cases}
\end{equation}
outside of the cut locus of $x$. Let us consider the above estimate for the family $r_\ve$ corresponding to the canonical variation $(\M,g_\ve)$. In this case, even though the left hand side of \eqref{eq:lapcompintro} converges as $\ve \to 0$ to the horizontal Laplacian of the sub-Riemannian distance $r_0$ from $x$, its right hand side tends to $+\infty$, yielding no information.

Let us stress that this is not a pathology of the canonical variation approximation scheme. In fact, it has been known for some time now that one cannot approximate in the Gromov-Hausdorff sense a sub-Riemannian structure through a sequence of Riemannian structures with the same dimension and (Ricci) curvature bounded from below \cite{Juillet1,Juillet2,RS-noCD}.

Nevertheless horizontal Laplacian comparison theorems do indeed hold, at least for particular classes of sub-Riemannian structures. For example, any complete $2d+1$-dimensional Sasakian structure with non-negative horizontal Tanaka-Webster curvature satisfies
\begin{equation}\label{eq:exampleintro}
\Delta_\hor r_0 \leq \frac{2d+2}{r_0},
\end{equation}
as proven in \cite{AAPL-3D,AAPL-Bishop,ZLL-Sasakian} in increasing degrees of generality, in terms of a measure contraction property essentially equivalent to \eqref{eq:exampleintro}. Notice that the aforementioned class of structures includes the Heisenberg groups and the Hopf fibrations, for which \eqref{eq:exampleintro} is sharp. 

Another hint comes from the class of $4d+3$-dimensional $3$-Sasakian structures. In this case, it has been proven in \cite{BR-BE} that, under suitable (possibly negative) curvature bounds, the estimate \eqref{eq:exampleintro} holds true with a constant equal to $4d+8$, which is sharply attained in the case of the quaternionic Hopf fibration.

\medskip

We remark however that all the aforementioned results are purely sub-Riemannian, and do not come with similar estimates for the whole family $\Delta_\hor r_\ve$, for $\ve > 0$. In \cite{BGKT17}, with methods similar to those used in the present paper, the authors filled this gap for the case of Sasakian foliations with non-negative horizontal Tanaka-Webster curvature. Among other results, they proved an uniform version of \eqref{eq:exampleintro}, namely
\begin{equation}\label{eq:exampleintro2}
\Delta_\hor r_\ve \leq \frac{N(\ve)}{r_\ve},
\end{equation}
valid for all $\ve \geq 0$, where $N(\ve)$ is a positive an asymptotically sharp constant, which reduces to \eqref{eq:exampleintro} as $\ve \to 0$. Notice that the validity of \eqref{eq:exampleintro2} for the canonical left-invariant metric on the Heisenberg group for all $\ve \geq 0$ was observed in \cite{RifCarnot} in terms of equivalent measure contraction properties, whereas the validity of analogous uniform estimates for general structures was left as an open problem.

\subsection*{Uniform comparison theorems}

The goal of this paper is to develop comparison theorems for the horizontal Laplacian of the distance from a given point $r_\ve(\cdot) :=d_\ve(x,\cdot)$, of the form
\begin{equation}\label{eq:inequalityintro}
\Delta_\hor r_\ve \leq F_\ve(r_\ve),  \qquad 
\forall \ve >0,
\end{equation}
which should hold outside of the cut locus, and under appropriate curvature-type assumptions. Here, $F_\ve : (0,\infty)\to \R$ is some continuous model function, depending on the particular class of structures under investigation. A fundamental requirement is that the right hand side of \eqref{eq:inequalityintro} admits a finite limit as $\ve \to 0$, yielding thus a truly sub-Riemannian comparison theorem for $r_0$. This is the case, for example, when the model function $F_\ve(\cdot)$ does not depend explicitly on $\ve$; we talk in this case of \emph{uniform} comparison theorems.

To achieve our goal, we make use of the variational theory of geodesics for the canonical variation of the metric, developed first in \cite{BGKT17} in the context of Sasakian foliations. This leads to a comparison principle, Theorem \ref{t:comparison}, which will be used throughout the paper. Such a theory can be applied a priori to any Riemannian structure $(\M,g)$ equipped with an orthogonal splitting. As it is natural to expect, when $g=g_\ve |_{\ve =1}$, further assumptions are required to ensure the uniformity of the result and, in turn, the existence of a sub-Riemannian limit.

\medskip

For this reason, in Section \ref{s:uniformcomparison}, we restrict our attention to the case of H-type foliations with parallel horizontal Clifford structure, which we briefly introduce here (cf.\ Section \ref{s:definitions} and \cite{BGMR-Htypefoliations}, where those structures were introduced, for further results). These are a general class of Riemannian foliations $(\M,g)$ with bundle-like metric and totally geodesic leaves. The tangent space to the leaves at each point defines the vertical sub-bundle $\ver$, and its orthogonal complement defines the horizontal sub-bundle $\hor$. Letting $J_Z: \Gamma(\hor)\to \Gamma(\hor)$ be the endomorphism defined by
\begin{equation}
g(J_Z X,Y) = g(Z,[Y,X]), \qquad \forall X,Y\in\Gamma(\hor),
\end{equation}
the H-type condition is the requirement that $J^2_Z =-\|Z\|^2 \mathbbold{1}_\hor$, for all $Z \in \Gamma(\ver)$. The parallel horizontal Clifford structure condition is instead an assumption on the covariant derivative of $J$, extending to more general foliations the fact that, for Sasakian structures, $J$ is parallel with respect to an appropriate metric connection. 

Let us mention that H-type foliations admitting a parallel horizontal Clifford structure have been classified in \cite{BGMR-Htypefoliations}, and include in particular Sasakian, $3$-Sasakian, negative $3$-Sasakian structures, H-type Carnot groups, and torus bundles over hyperk\"ahler manifolds.

We introduce a decomposition of the tangent space $T_{\gamma(t)}\M$ along any minimizing $g_\ve$-geodesic, such that the restriction of the horizontal Hessian of $r_\ve$ on these subspaces is controlled by suitable curvature invariants, uniformly bounded as $\ve \to 0$, cf.\ Section \ref{s:splitting}. An important role is played by the natural Bott connection $\nabla$ of the foliation, and an associated family of metric connections $\hat{\nabla}^\varepsilon$, first introduced in \cite{Bau14,BKW}. In this setting we obtain a class of \emph{uniform} comparison theorems for the Hessian of $r_\ve$, along the subspaces of the aforementioned decomposition, cf.\ Theorems \ref{t:geoddir}, \ref{t:Riemannian-sect}, \ref{t:Riemannian-average}, \ref{t:SasakianZpar}, and \ref{t:SasakianZperp}. As a consequence we obtain a uniform horizontal Laplacian comparison theorem for $r_\ve$, cf.\ Theorem \ref{t:horlapuniform}, and its corresponding sub-Riemannian limit, cf.\ Theorem \ref{t:horlapuniform0}.

\medskip

In Section \ref{s:sharpBM} we study a more general class of structures, which are not necessarily foliations, but still satisfy the H-type condition. In this greater generality we obtain a sharp sub-Riemannian Bonnet-Myers comparison theorem valid only in the limit $\ve \to 0$, cf.\ Theorem \ref{t:BMII} . The latter, in particular, recovers and extends similar results obtained in the recent literature, with an unified and simpler proof. We refer in particular to \cite[Thm.\ 1.7]{BRcontact}, for contact structures \cite[Thm.\ 1]{BI17} for quaternionic contact structures, \cite[Cor.\ 9]{RS-3sas} for $3$-Sasakian structures, all of which are based on the intrinsic comparison theory developed in \cite{BR-comparison}. See also \cite[Sec.\ 3.3]{Hug95} for similar Bonnet-Myers type results for three-dimensional contact structures. Let us also note that from Corollary 2.21 in \cite{BGMR-Htypefoliations},  a non-sharp Bonnet-Myers theorem on any H-type foliation may be obtained under weaker conditions (a positive lower bound on the horizontal Ricci curvature) by using the Bochner's method based on the theory of generalized curvature dimension inequalities developed in \cite{BG17}.

\medskip

In order to pass to the limit for $\ve \to 0$ in the aforementioned results, we prove the $C^\infty$ convergence of $d_\ve \to d_0$ outside of the sub-Riemannian cut locus which, to our knowledge, is new and of independent interest (cf.\ Proposition \ref{p:cutapprox}).

\paragraph*{Acknowledgements}
F.\ B.\ was supported in part by the NSF grants DMS-1660031, DMS-1901315 and the Simons Foundation. S.\ V.-M.\ was supported in part by the NSF grant DMS-1901315 (PI F. Baudoin). E.\ G.\ was supported in part by the Research Council of Norway (project number 249980/F20). L.\ R.\ was supported by the Grants ANR-15-CE40-0018 and ANR-18-CE40-0012 of the ANR, and by a public grant as part of the Investissement d'avenir project, reference ANR-11-LABX-0056-LMH, LabEx LMH, in a joint call with the ``FMJH Program Gaspard Monge in optimization and operation research''. This project has received funding from the European Research Council (ERC) under the European Union’s Horizon 2020 research and innovation programme (grant agreement GEOSUB, No. 945655)

\section{Preliminaries: H-type sub-Riemannian manifolds}\label{s:definitions}

\subsection{Hladky connection and canonical variation}

We consider a general Riemannian structure $g$ on $\M$ together with a vector bundle orthogonal splitting $T\M = \hor \oplus \ver$. The sub-bundles $\hor$ and $\ver$ are referred to as the set of horizontal and vertical direction, respectively. In this setting there exists a canonical $g$-metric connection $\nabla$ preserveing the splitting, the connection defined by Hladky in \cite{Hladky}.

\begin{proposition}[\cite{Hladky}]
There exists a unique metric connection $\nabla$ on $\M$ such that:
\begin{itemize}
\item $\mathcal{H}$ and $\mathcal{V}$ are $\nabla$-parallel, i.e. for every $X \in \Gamma(\hor), Y \in \Gamma(T\M), Z \in \Gamma (\V)$,
\begin{equation}
\nabla_Y X \in \Gamma(\hor), \quad \nabla_Y Z \in \Gamma(\V).
\end{equation}
\item The torsion $T$ of $\nabla$ satisfies $T(\hor,\hor) \subset \V$ and $T(\V,\V) \subset \hor$.
\item For every $X,Y \in \Gamma(\hor)$, $V,Z \in \Gamma(\V)$, 
\begin{equation}
\langle T(X,Z),Y \rangle =\langle T(Y,Z),X \rangle, \qquad   \langle T(Z,X), V \rangle =\langle T(V,X), Z \rangle.
\end{equation}
\end{itemize}
\end{proposition}
The \emph{canonical variation} of the metric, introduced in \cite[Ch.\ 9]{Besse} in the context of submersions, is the one-parameter family of Riemannian metrics defined by
\begin{align}\label{cv}
g_{\varepsilon}=g_\mathcal{H} \oplus  \frac{1}{\varepsilon }g_{\mathcal{V}}, \qquad \varepsilon >0,
\end{align}
where $g_{\mathcal{H}}$ and $g_{\mathcal{V}}$ denote the restriction of the Riemannian metric $g$ to the horizontal and vertical sub-bundle, respectively.

For any $X \in \Gamma(T\M)$, the symbol $\mathcal{L}_X$ denotes the Lie derivative in the direction of~$X$. We write $X_\hor = \pi_\hor (X) \in \Gamma(\hor)$ and $X_\ver =\pi_\ver(Y) \in \Gamma(\ver)$ for the horizontal and vertical projections, respectively.
We define the $(2,1)$ tensor $A$ by the formula:
\begin{equation}
g(A_X Y, Z) = \frac{1}{2} (\mathcal{L}_{X_\ver} g)(Y_\hor,Z_\hor) + \frac{1}{2}(\mathcal{L}_{X_\hor}g)(Y_\ver,Z_\ver).
\end{equation}
The following properties hold:
\begin{equation}
A_\ver \ver = 0,\qquad A_{\ver} \hor \subseteq \hor, \qquad A_\hor \hor = 0,\qquad A_{\hor} \ver \subseteq \ver.
\end{equation}

The Hladky connection can be expressed in terms of the Levi-Civita one $\nabla^g$ by
\begin{equation}
\nabla_X Y =
\begin{cases}
\pi_{\mathcal{H}} ( \nabla_X^{g} Y) &  X,Y \in \Gamma(\mathcal{H}), \\
\pi_{\mathcal{H}} ( [X,Y]) +A_X Y & X \in \Gamma(\mathcal{V}), \quad  Y \in \Gamma(\mathcal{H}), \\
\pi_{\mathcal{V}} ( [X,Y]) + A_X Y & X \in \Gamma(\mathcal{H}),\quad  Y \in \Gamma(\mathcal{V}), \\
\pi_{\mathcal{V}} ( \nabla_X^{g} Y) & X,Y \in \Gamma(\mathcal{V}).
\end{cases}
\end{equation}
Finally, the torsion $T$ of the Hladky connection $\nabla$ is given by
\begin{equation}
T(X,Y) = \begin{cases}
- \pi_\ver([X,Y]) & X,Y \in \Gamma(\hor), \\
A_X Y - A_Y X &  X \in \Gamma(\hor), Y \in \Gamma(\V), \\
- \pi_\hor([X,Y]) & X,Y \in \Gamma(\ver).
\end{cases}
\end{equation}
The above formulas show that the Hladky connection defined relative to $g_\ve$ in \eqref{cv} will coincide for all choices of $\ve >0$.

\begin{definition}
The complement $\V$ is called metric if $A=0$.
\end{definition}

Let $(\M,g)$ be a Riemannian manifold, equipped with an orthogonal splitting $T\M = \hor \oplus \ver$. If $\ver$ is integrable and metric, then $(\M,\hor,g)$ is a Riemannian foliation with bundle-like metric and totally geodesic leaves, tangent to $\ver$. We simply refer to these structures as \emph{totally geodesic foliations}.  In this foliation context, Hladky connection is referred to as the Bott connection (see \cite{BGMR-Htypefoliations}).

\subsection{The H-type and the \texorpdfstring{$J^2$}{J2} conditions}

We introduce a condition  that will play a prominent role in the following. For any $Z \in T\M$, let $J_Z :  T\M \to T\M$ be defined as 
\begin{equation} \label{Jmap}
\langle J_Z X,Y\rangle = \langle Z,T(X,Y)\rangle, \qquad \forall X,Y \in T\M.
\end{equation}
We remark that $J$ is defined for any Riemannian manifold $(\M,g)$ equipped with an orthogonal splitting $T\M = \hor \oplus \ver$.
\begin{remark}
If $(\M,g)$ is a totally geodesic foliation, it holds
\begin{equation}
J_\hor =0, \qquad J_\ver \ver = 0,\qquad J_{\ver} \hor \subseteq \hor.
\end{equation}
\end{remark}

\begin{definition}
We say that the H-type condition is satisfied if 
\begin{equation}
J_Z^2 = - \| Z \|^2 \mathbbold{1}_{\mathcal{H}}, \qquad \forall Z \in \Gamma(\ver).
\end{equation}
\end{definition}

\begin{definition}[\cite{Cowling-Htype,CalinChangMarkina}]
We say that $J^2$ condition holds if for all $Z,Z'\in \Gamma(\ver)$, $X \in \Gamma(\hor)$ with $\langle Z,Z'\rangle =0$ there exists $Z''\in \Gamma(\ver)$ such that
\begin{equation}
J_Z J_{Z'}X = J_{Z''}X.
\end{equation}
\end{definition}
We stress that $Z''$ depends on $Z,Z'$ but may also depend on $X$. This condition is true in particular if the vector space generated by $\mathbbold{1}_\hor$ and the $J_Z$, for $Z \in \Gamma(\ver)$ is a subalgebra.

We then recall the following definitions from \cite{BGMR-Htypefoliations}.

\begin{definition}(\cite{BGMR-Htypefoliations}) \label{def:HType}
A totally geodesic foliation $(\M,\hor,g)$ for which the H-type condition holds is called an H-type foliation. Moreover:
\begin{enumerate}[\rm (i)]
\item If the torsion of the Bott connection is horizontally parallel, i.e.\ $\nabla_\hor T=0$, then we say that $(\M,\hor,g)$ is an H-type foliation with horizontally parallel torsion.
\item If the torsion of the Bott  connection is completely parallel, i.e.\ $\nabla T=0$, then we say that $(\M,\hor,g)$ is an H-type foliation with parallel torsion.
\item Let $(\M,\hor,g)$ be an H-type foliation with horizontally parallel torsion. We say that  $(\M,\hor, g)$  is an H-type foliation with a parallel horizontal Clifford structure if there exists a constant
$\kappa \in \mathbb{R}$ such that for every $Z_1,Z_2 \in \Gamma{\ver}$
\begin{align} \label{CliffordKappa}
(\nabla_{Z_1} J)_{Z_2}=-\kappa \left( J_{Z_1}J_{Z_2} +\langle Z_1,Z_2 \rangle_\V  \mathbbold{1}_\hor \right).
\end{align}
\end{enumerate} 

\end{definition}

Several examples of manifolds satisfying the previous definitions are given in \cite{BGMR-Htypefoliations}. Since H-type foliations with a parallel horizontal Clifford structure and satisfying $J^2$ condition will play an important role in the next section, we point out some examples that satisfy these assumptions in Table \ref{Table 1} and refer to \cite{BGMR-Htypefoliations} for further details on such spaces.

\begin{table}[t]
\centering
\scalebox{0.8}{
\begin{tabular}{| l | l | l |}
\hline \textbf{Complex Type} \\ \hline
Sasakian manifolds  \\ \hline
\hline \textbf{Quaternionic Type} \\ \hline
3-Sasakian manifolds \\ \hline
Negative 3-Sasakian  manifolds  \\ \hline
Torus bundle over hyperk\"ahler manifolds  \\ \hline
\hline \textbf{Octonionic Type} \\ \hline
Octonionic Heisenberg Group   \\ \hline
Octonionic Hopf Fibration  \(\mathbb{S}^7 \hookrightarrow \mathbb{S}^{15} \to \mathbb{O}P^1\) \\ \hline
Octonionic Anti de-Sitter Fibration  \(\mathbb{S}^7 \hookrightarrow \mathbf{AdS}^{15}(\mathbb{O}) \to \mathbb{O}H^1\)  \\ \hline
\end{tabular}}
\caption{Some examples of H-type foliations with parallel horizontal Clifford structure and satisfying the $J^2$ condition.}\label{Table 1}
\end{table}

\subsection{The sub-Riemannian limit}

Let $(\M,g)$ be a Riemannian manifold equipped with an orthogonal splitting $T\M = \hor \oplus \ver$, and assume that $\hor$ is bracket-generating. Let $\{g_\ve\}_{\ve >0}$ be the canonical variation \eqref{cv}. The Riemannian distance associated with $g_{\varepsilon}$ will be denoted by $d_\varepsilon$, while the sub-Rieman\-nian one, which depends only on the restriction $g_\hor$ of $g$ to $\hor$, is $d_0$. In this section we discuss how $d_\varepsilon$ approximates $d_0$ as $\varepsilon \to 0$ in full generality.

We always assume that $(\M,d_1)$ is complete, and so the same necessarily holds for $(\M,d_\ve)$ for all $\varepsilon \geq 0$. The cut locus $\Cut_\ve (x)$  of  $x\in \M$, $\ve \geq 0$ for the distance $d_\ve$ is defined as the complement of the set of points $y\in\M$ such that there exists a unique length-minimizing normal geodesic joining $x$ and $y$, and its endpoints are not conjugate (see \cite{A}). The global cut locus of $\M$ is defined by
\begin{equation}
\Cut_\ve (\M)=\left\{ (x,y) \in \M \times \M,\  y \in \Cut_\ve (x) \right\}.
\end{equation}

\begin{lemma}[\cite{A}, \cite{RT}]\label{cutlocus}
Let $\ve \ge 0$. The following statements hold:
\begin{enumerate}[\rm 1.]
\item The set $\M \setminus \Cut_\ve (x_0)$ is open and dense in $\M$.
\item The function $(x,y) \mapsto d_\ve  (x,y)^2$ is smooth on $\M \times \M \setminus \Cut_\ve (\M)$.
\end{enumerate}
\end{lemma}

The main result of this section is the following proposition, establishing the $C^\infty$-convergence of $d_\ve$ to $d_0$ outside of the cut locus. We will use this result to obtain the sub-Riemannian limit of the uniform horizontal Laplacian comparison theorems. The proof is included in Appendix \ref{b:app} since it is rather long and the techniques we make use of will not be reused in the sequel of this paper.

\begin{proposition}\label{p:cutapprox}
Let $x,y \in \M$ with $y \notin \Cut_0(x)$. Then there exists an open neighbourhood $V$ of $y$ and $\varepsilon'>0$ such that $V \cap \Cut_\ve(x) = \emptyset$ for all $0\leq \ve < \ve'$. Furthermore, the map
\begin{equation}
(\ve,z) \mapsto r_\ve(z)=d_\ve(x,z)
\end{equation}
is smooth for $(\ve,z)\in [0, \ve')\times V$. In particular, we have uniform convergence $r_\ve \to r_0$ together with their derivatives of arbitrary order on compact subsets of $\M \setminus \Cut_0(x)$.
\end{proposition}

\begin{remark}\label{rmk:convergence}
For $x\in \M$, let $r_\ve(\cdot) := d_\ve(x,\cdot)$, and let $y\notin \Cut_0(x)$. The $C^\infty$ convergence established in Proposition \ref{p:cutapprox} implies
\begin{equation}
\lim_{\ve \to 0} \nabla_\hor r_\ve = \nabla_\hor r_0, \qquad \lim_{\ve \to 0} \nabla_\ver r_\ve = \nabla_\ver r_0,
\end{equation}
outside of $\Cut_0(x)$, where $\nabla_\hor$ and $\nabla_\ver$ denote the projections on $\hor$ and $\ver$ of the Riemannian gradient $\nabla^g$. Furthermore, since $\nabla^{g_\ve} = \nabla^g_\hor + \ve \nabla^g_\ver$, we have that for all $\ve>0$ and outside of $\Cut_0(x)$ it holds
\begin{equation}
1 = \|\nabla^{g_\ve} r_\ve\|^2_\ve = \|\nabla_\hor r_\ve \|^2 + \ve \|\nabla_\ver r_\ve\|^2,
\end{equation}
and thus $\displaystyle\lim_{\ve \to 0} \|\nabla_\hor r_\ve \| =1$.
\end{remark}

\subsubsection{H-type structures are ideal}

The H-type condition implies that the horizontal distribution $\hor$ is \emph{fat} (also said \emph{strong bracket-generating}), that is for any $x\in \M$ and section $X\in \Gamma(\hor)$ with $X(x)\neq 0$ it holds $T_x \M = \hor_x \oplus [X,\hor]_x$. It is well known that sub-Riemannian structures with a fat distribution are \emph{ideal}, that is do not admit non-trivial singular minimizing geodesics, see e.g.\ \cite{Rifford,Montgomery}. A well-known consequence of this fact is that the squared sub-Riemannian distance is locally semi-concave outside of the diagonal, see \cite{CR-Semiconcavity}. In particular, the sub-Riemannian squared distance from $x$ is locally Lipschitz in charts on $\M\setminus \{x\}$. It follows by standard arguments (see e.g.\ \cite{Rifford}) that in this case $\Cut_0(x)$ has zero measure for all $x\in \M$. We state this result as a proposition.

\begin{proposition}
Any complete H-type sub-Riemannian structure is fat, and in particular the set $\Cut_\ve(x)$ has zero measure for all $x \in \M$ and $\ve \geq 0$.
\end{proposition}

We remark that, for all $\ve >0$, $\Cut_\ve(x)$ can be characterized as the set of points where $d_\ve(x,\cdot)$ fails to be locally semi-convex \cite{CEMS10}, and the same characterization holds for $\Cut_0(x)$, provided that the sub-Riemannian structure is ideal \cite{BR-Interpolation}. This is the case, as we already remarked, for H-type structures.

\subsection{The comparison principle}

Let $\nabla$ be the Hladky connection, which is $g_\varepsilon$-metric for all $\varepsilon>0$. In Appendix \ref{a:app} we show how one can build an associated metric connection, with metric adjoint (see Lemma \ref{l:construction}). Following the notation of \cite{BGKT17}, we denote such a connection by
\begin{equation}\label{eq:hatnablaeps}
\hat{\nabla}^\varepsilon_X Y  = \nabla_X Y + J^\varepsilon_X Y, \qquad \forall X,Y\in\Gamma(T\M),
\end{equation}
where $J$ and $J^\ve$ are defined as in \eqref{Jmap} relative to $g$ and $g_\ve$ respectively. Its adjoint connection (cf.\ Appendix), which is also $g_\ve$-metric, is given by
\begin{equation}
\nabla^\ve_X Y = \nabla_X Y - T(X,Y) + J_Y^\ve X, \qquad \forall X,Y\in\Gamma(T\M).
\end{equation}
Let $\hat{R}^\ve$ be the curvature of $\hat{\nabla}^\varepsilon$. The Jacobi equation for a vector field $W$ along a $g_\varepsilon$-geodesic $\gamma$ reads
\begin{equation}
\hat\nabla^\varepsilon_{\dot\gamma} \nabla^\varepsilon_{\dot\gamma}W+ \hat{R}^{\varepsilon}(W,\dot\gamma)\dot\gamma = 0.
\end{equation}
The following comparison principle will be repeatedly used in the following. It is proved in Appendix \ref{a:app}, in a slightly more general context (take $D =\hat{\nabla}^\varepsilon$ there, so that $\hat{D} = \nabla^\ve$).

\begin{theorem}[Comparison principle]\label{t:comparison}
Let $(\M,g)$ be a Riemannian manifold, together with a vector bundle orthogonal splitting $T\M = \hor \oplus \ver$. Fix $\varepsilon>0$ and let $g_\ve$ be the corresponding one-parameter family of Riemannian metrics as in \eqref{cv}. Choose $x \in \M$ and $y \notin \Cut_\varepsilon(x)$. Let $\gamma : [0,r_\varepsilon] \to \M$ be the unique $g_\varepsilon$-geodesic, parametrized with unit speed, joining $x$ with $y$. For $\ell \in \mathbb{N}$, let $W_1,\dots,W_\ell$ be $\ell$-tuple of vector fields along $\gamma$ and $g_\varepsilon$-orthogonal to $\dot\gamma$ such that
\begin{equation}\label{eq:comparison-condition}
\sum_{i=1}^\ell \int_0^r \langle\hat\nabla^\varepsilon_{\dot\gamma} \nabla^\varepsilon_{\dot\gamma}W_i+ \hat{R}^{\varepsilon}(W_i,\dot\gamma)\dot\gamma,W_i\rangle_\varepsilon \, dt \geq 0.
\end{equation}
Then, at $y=\gamma(r_\varepsilon)$, it holds
\begin{equation}\label{eq:pretraceI}
\sum_{i=1}^\ell \mathrm{Hess}^{\hat{\nabla}^\varepsilon}(r_\varepsilon)(W_i,W_i) \leq \sum_{i=1}^\ell \langle W_i(r_\varepsilon),\hat{\nabla}^\varepsilon_{\dot\gamma}W_i(r_\varepsilon)\rangle_\ve,
\end{equation}
and the equality holds if and only if $W_1,\dots,W_\ell$ are Jacobi fields for the metric $g_\ve$.
\end{theorem}
\begin{remark}\label{rmk:hessian}
If $(\M,\hor,g)$ is a totally geodesic foliation, $W\in \Gamma(\hor)$ and $u$ is sufficiently regular, then one has
\begin{multline}
\mathrm{Hess}^{\hat{\nabla}^\varepsilon}(u)(W,W) =\langle \hat{\nabla}^\varepsilon_W \nabla^{g_\ve} u ,W\rangle_\ve = \langle (\nabla_W +J_W^\ve)\nabla^{g_\ve} u,W\rangle_\ve \\ 
=\langle \nabla_W \nabla u,W\rangle =: \mathrm{Hess}(u)(W,W).
\end{multline}
In particular, if the $W_i$ in Theorem \ref{t:comparison} are assumed to be horizontal at $y = \gamma(r_\ve)$, then the Hessian in \eqref{eq:pretraceI} can be computed equivalently with respect to $\nabla$.
\end{remark}

The following Lemma contains some useful facts that will be used in computations.
\begin{lemma}\label{l:someproperties}
Let $(\M,\hor,g)$ be a totally geodesic foliation. Then for all $\varepsilon >0$ it holds:
\begin{equation}\label{eq:nablaJ}
(\hat{\nabla}^\varepsilon_{X} J)_Y Z = (\nabla_{X} J)_Y Z + [J_{X}^\varepsilon,J_Y]Z, \qquad \forall X,Y,Z \in \Gamma(T\M).
\end{equation}
If the torsion is horizontally parallel, then
\begin{equation}
(\hat{\nabla}^\varepsilon_Z J)_W = -(\hat{\nabla}^\varepsilon_Z J)_W \quad \text{and} \quad (\nabla_Z J)_W = -(\nabla_Z J)_W, \qquad  \forall Z,W \in \Gamma(T\M).
\end{equation}
If the H-type condition holds, then for any $X,Y\in \Gamma(T\M)$, and $Z\in \Gamma(\ver)$, we have 
\begin{gather}
\langle J_ZY, (\nabla_X J)_Z Y \rangle  = \langle J_ZY, (\hat{\nabla}^\varepsilon_X J)_Z Y\rangle = 0 \label{eq:perpendicularity} \\
T(J_Z X,X) = -\|X_\hor\|^2 Z \label{eq:propT} \\
(\nabla_{Y} T)(J_Z X,X)  = - T((\nabla_{Y}J)_ZX,X). \label{eq:nablaTnablaJ} 
\end{gather}
\end{lemma}
\begin{proof}
The first identity follows from the definition $\hat{\nabla}^\varepsilon_X Y = \nabla_X Y + J^\varepsilon_X Y$, and the properties of $J$ in the totally geodesic setting. The second identity is proved for $\nabla$ in \cite[Lemma 2.6]{BGMR-Htypefoliations}. Then, thanks to the first identity it holds also for $\hat{\nabla}^\varepsilon$. To prove \eqref{eq:perpendicularity}, we stress that both $\nabla$ and $\hat{\nabla}^\varepsilon$ connections are metric and preserve the splitting $\hor\oplus\ver$. Therefore \eqref{eq:perpendicularity} can be easily proved at any point $p \in \M$ by assuming without loss of generality that $Z$ and $Y$ are parallel with respect to  $\nabla$ or $\hat{\nabla}^\varepsilon$. Identity \eqref{eq:propT} is trivial, and \eqref{eq:nablaTnablaJ} follows by taking the covariant derivative of \eqref{eq:propT}.
\end{proof}

In order to verify condition \eqref{eq:comparison-condition} of Theorem \ref{t:comparison} it is useful highlight the role of the curvature of the Hladky connection, writing explicitly the Jacobi operator. In the next lemma we do it for the case of H-type foliations with horizontally parallel torsion.

\begin{lemma}\label{l:Jacobiequation}
Let $(\M,\hor,g)$ be an H-type foliation with horizontally parallel torsion. Let $W$ be a vector field along a $g_\varepsilon$-geodesic $\gamma$ with $\ve>0$. Let $\hat R^\ve$ be the curvature of $\hat \nabla^\ve$ and define the Jacobi operator
\begin{equation} \label{JacobiOperator}
\mathcal{Z}(W):=\hat\nabla^\varepsilon_{\dot\gamma} \nabla^\varepsilon_{\dot\gamma}W+ \hat{R}^{\varepsilon}(W,\dot\gamma)\dot\gamma.
\end{equation}
\begin{enumerate}[\rm (a)]
\item Let $R$ be the curvature of the Hladky (Bott) connection $\nabla$ and define
\begin{equation}
R_\hor(X, Y) Z = R(X_\hor, Y_\hor) Z_\hor, \qquad R_\V(X, Y) Z = R(X_\ver, Y_\ver) Z_\ver.
\end{equation}
Then one has, for the Jacobi operator \eqref{JacobiOperator}, the formula
\begin{align}
\mathcal{Z}(W) & = \hat{\nabla}^\varepsilon_{\dot\gamma}\hat{\nabla}^\varepsilon_{\dot\gamma} W
 -J_{\dot\gamma}^\varepsilon \hat{\nabla}^\varepsilon_{\dot\gamma}W + J^\varepsilon_{\hat{\nabla}^\varepsilon_{\dot\gamma} W}\dot\gamma - (\nabla_{\dot\gamma} J^\varepsilon)_W\dot\gamma + J^\varepsilon_{T(W,\dot\gamma)}\dot\gamma +R_\hor(W,\dot\gamma)\dot\gamma \\
 & \qquad  +  \hat{\nabla}^\varepsilon_{\dot\gamma}(T(W,\dot\gamma)) + (\nabla_{\dot\gamma}T)(W,\dot\gamma) +R_\ver(W,\dot\gamma)\dot\gamma.
\end{align}
\item Assume that  $(\M,\hor,g)$ admits a parallel horizontal Clifford structure as in Definition~\ref{def:HType}~{\rm (iii)} with constant $\kappa$. Let $W_\perp$ be the $g_\ve$-orthogonal projection of $W$ on the orthogonal complement of $\dot\gamma$. Then
\begin{align}
\mathcal{Z}(W) & = \hat{\nabla}^\varepsilon_{\dot\gamma}\hat{\nabla}^\varepsilon_{\dot\gamma} W
 -J_{\dot\gamma}^\varepsilon \hat{\nabla}^\varepsilon_{\dot\gamma}W + J^\varepsilon_{\hat{\nabla}^\varepsilon_{\dot\gamma} W}\dot\gamma +\kappa J_{\dot\gamma} J_{(W_{\V})_\perp}^\varepsilon\dot\gamma
 + J^\varepsilon_{T(W,\dot\gamma)}\dot\gamma +R_\hor(W,\dot\gamma)\dot\gamma \\
 & \qquad +\hat{\nabla}^\varepsilon_{\dot\gamma}(T( W,\dot\gamma)) + \kappa (T(J_{\dot\gamma}W,\dot\gamma)+\langle W,\dot\gamma_\hor \rangle\dot\gamma_\ver) +\kappa^2 \|\dot\gamma_\ver\|^2 (W_\ver)_\perp.
\end{align}
\end{enumerate}
\end{lemma}
\begin{proof}
Recall that $\nabla^\varepsilon_{\dot\gamma}W = \hat{\nabla}^\varepsilon_{\dot\gamma}W -\hat{T}^\varepsilon(\dot\gamma,W)$, hence the Jacobi operator reads
\begin{equation}\label{eq:Jacobiproof}
\mathcal{Z}(W) = \hat{\nabla}^\varepsilon_{\dot\gamma}(\hat{\nabla}^\varepsilon_{\dot\gamma}W -\hat{T}^\varepsilon(\dot\gamma,W)) + \hat{R}^\varepsilon(W,\dot\gamma)\dot\gamma.
\end{equation}
The proof follows by explicit computation of the horizontal and vertical part of the above equation. We just highlight some non-trivial fact that simplify the computation. First, using the definition of $\hat{\nabla}^\varepsilon$, we obtain the general formula
\begin{equation}
\hat{R}^\varepsilon(W,X)X = R(W,X)X+(\nabla_W J^\varepsilon)_{X} X -(\nabla_{X} J^\varepsilon)_W X  +  J^\varepsilon_{T(W,X)}X + [J_W^\varepsilon,J_{X}^\varepsilon]X.
\end{equation}
Furthermore, by \cite[Lemma 2.7]{BGMR-Htypefoliations}, one has
\begin{align*}
R(W,X)X  = R_\hor(W,X)X +R_\ver(W,X)X  +  (\nabla_{X} T)(W,X).
\end{align*}
This allows to deal with the last term of \eqref{eq:Jacobiproof}. To deal with the first two terms of \eqref{eq:Jacobiproof} we use instead	
\begin{equation}
\hat{T}^\varepsilon(X,Y)  =  T(X, W) + J^\varepsilon_{X} W - J_W^\varepsilon X,
\end{equation}
and that $(\hat{\nabla}_{V} J)_V =(\nabla_{V} J)_V = 0$ for all $V\in \Gamma(\ver)$.

In the case of a parallel horizontal Clifford structure, recall that by \cite[Thm.\ 3.6]{BGMR-Htypefoliations} the leaves of the vertical foliation have constant (Bott) curvature equal to $\kappa^2$. Observe that, in terms of torsion, the defining condition \eqref{CliffordKappa} reads
\begin{equation}
(\nabla_V T)(X,Y) = \kappa \left( T(J_V X,Y) + \langle X,Y\rangle V\right), \qquad \forall X,Y \in \Gamma(\hor),\, V \in \Gamma(\ver).
\end{equation}
The proof is complete.
\end{proof}

\section{Uniform comparison theorems for H-type foliations with parallel horizontal Clifford structure}\label{s:uniformcomparison}

In this section $(\M,\hor,g)$ is an H-type foliation with parallel horizontal Clifford structure and satisfying the $J^2$ condition. Some of the results proved in this section hold under slightly more general hypotheses; we refer to the statements for the exact required assumptions. Throughout the section, we will denote $n=\rank \mathcal{H}$ and $m=\rank \mathcal V$ and make use of the following notation for the comparison functions:

\begin{equation}
F_{\mathrm{Rie}}(r,k) = \begin{cases}  \sqrt{k} \cot \sqrt{k} r & \text{if $k > 0$,} \\
\frac{1}{r} & \text{if $k = 0$,}\\ \sqrt{|k|} \coth \sqrt{|k|} r & \text{if $k < 0$,} \end{cases}
\end{equation}
and 
\begin{equation}
F_{\mathrm{Sas}}(r,k) = \begin{cases}  \frac{\sqrt{k}(\sin \sqrt{k}r  -\sqrt{k} r \cos \sqrt{k} r)}{2 - 2\cos \sqrt{k} r - \sqrt{k} r \sin \sqrt{k} r} & \text{if $k > 0$,} \\
\frac{4}{r} & \text{if $k = 0$,}\\ \frac{\sqrt{|k|}( \sqrt{|k|} r \cosh \sqrt{|k|} r - \sinh \sqrt{|k|}r)}{2 - 2\cosh \sqrt{|k|} r + \sqrt{|k|} r \sinh \sqrt{|k|} r} & \text{if $k < 0$.} \end{cases}
\end{equation}

\subsection{The splitting}\label{s:splitting}

The following splitting which is independent from $\ve$ plays a crucial role for the analysis of the index form. Fix a vector field $Y\in \Gamma(T\M)$ with $Y_\ch \neq 0$. Usually $Y=\dot\gamma$ is the tangent vector to a geodesic, in which case the definition that follows makes sense along $\gamma$. We define the splitting
\begin{equation}\label{eq:splittinghor}
\hor = \hor_{\sas}(Y) \oplus \hor_{\rie}(Y) \oplus \spn \{Y_\hor\},
\end{equation}
where each subspace is defined as
\begin{align}
\hor_{\sas}(Y) &= \{J_Z Y \mid Z \in \ver\}, \\
\hor_{\rie}(Y) &= \{X \in \hor \mid X \perp \hor_{\sas}(Y)\oplus  \spn \{Y_\hor\} \}.
\end{align}
\begin{remark}
If $Y \in T\M \setminus \V$, one can check that $\dim \hor_{\sas}(Y) = m$, and $\dim \hor_\rie (Y)= n-m-1$, and the splitting depends only on the horizontal part of $Y$. If the $J^2$ condition holds, for any non-zero $Z\in\ver$, the non-degenerate and skew-symmetric operator $J_Z:\hor \to \hor$ preserves $\hor_\rie(Y)$, and hence $m$ is odd and the dimension of $\hor_\rie$ is even.
\end{remark}

The next proposition motivates the relevance and the naturalness of the splitting in relation with the connection $\hat{\nabla}^\varepsilon$. We say that the splitting \eqref{eq:splittinghor} associated with a curve $Y=\dot\gamma$ is $\hat{\nabla}^\ve_{\dot\gamma}$-parallel if each sub-bundle composing it is $\hat{\nabla}^\ve_{\dot\gamma}$-parallel.
\begin{proposition}\label{p:parallelism}
Let $(\M,\hor,g)$ be an H-type foliation with parallel horizontal Clifford structure, satisfying the $J^2$ condition. Let $\gamma:[0,r_\varepsilon] \to \M$ be a $g_\varepsilon$-geodesic with $\dot \gamma_\hor \neq 0$. Then the splitting \eqref{eq:splittinghor} associated with $\dot\gamma$ is orthogonal and $\hat{\nabla}^\ve_{\dot\gamma}$-parallel.
\end{proposition}
\begin{proof}
The orthogonality of the horizontal splitting is immediate from the definition. Since the connection is metric and $\hat \nabla_{\dot \gamma }^\ve \dot \gamma = 0$,  it is sufficient to prove that $\hor_{\sas}(\dot \gamma)$ is parallel to complete the result. First we observe that $\hat \nabla_{\dot \gamma}^\ve J_{\dot \gamma_\V} \dot \gamma = 0$. Furthermore, for any $\hat \nabla^\ve$-parallel vector field $Z(t)$ with values in $\V$ and orthogonal to $\dot \gamma_\V$, we have
\begin{equation}
\hat \nabla_{\dot \gamma}^\ve J_Z\dot \gamma = \frac{2- \ve \kappa}{\ve} J_{\dot \gamma} J_Z \dot \gamma.
\end{equation}
From the $J^2$ condition, it follows that $\hor_{\sas}(\dot \gamma)$ has a basis of vector fields with covariant derivatives in itself and is hence parallel.
\end{proof}


\paragraph{Notation.} We introduce some notation that will be used in the forthcoming sections. For a $g_\varepsilon$-geodesic $\gamma :[0,r_\varepsilon] \to \M$, joining $x$ with $y \notin \Cut_\varepsilon(x)$, one has $\dot\gamma = \nabla^{g_\varepsilon}r_\varepsilon$, so that
\begin{equation}
\dot\gamma = \dot\gamma_\hor + \dot\gamma_\ver= \nabla_\hor r_\varepsilon + \varepsilon \nabla_\ver r_\varepsilon.
\end{equation}
If the geodesic has unit speed, one has $\|\dot\gamma(t)\|_\varepsilon=1$. Defining 
\begin{equation}
h := \|\nabla_\hor r_\varepsilon\|, \qquad v := \|\nabla_\ver r_\varepsilon\|,
\end{equation}
one has that $h$ and $v$ are constant along the geodesic, and the following relations hold:
\begin{equation}
\|\dot\gamma_\hor\|^2_\varepsilon = h^2, \qquad \|\dot\gamma_\ver\|^2_\varepsilon = \varepsilon v^2, \qquad h^2+\varepsilon v^2 =1.
\end{equation}
We recall that $\mathrm{Hess}$ denotes the Hessian with respect to the reference Bott connection $\nabla$.

\subsection{Comparison in the direction of the geodesic}

The study of $\mathrm{Hess}(r_\varepsilon)$ in the direction of the geodesic joining $x$ with $y \notin \Cut_\varepsilon(x)$ is elementary, as a consequence of the eikonal equation $\| \nabla^{g_\ve} r_\ve \|_\ve = 1$. For the case of horizontal projection, some care is necessary if one wants a result uniform with respect to  $\varepsilon>0$. We give here a self-contained and somewhat sharper statement than the one in \cite[Thm.\ 3.5 (1)]{BGKT17}.

\begin{theorem}\label{t:geoddir}
 Let $(\M,\hor,g)$ a totally geodesic foliation. Then, at all points $y\notin \Cut_\ve(x)$ with $\nabla_\hor r_\ve(y) \neq 0$, it holds
 \begin{equation}
 \mathrm{Hess}(r_\varepsilon)\left(\frac{\nabla_\hor r_\varepsilon}{\|\nabla_\hor r_\varepsilon\|},\frac{\nabla_\hor r_\varepsilon}{\|\nabla_\hor r_\varepsilon\|}\right) \leq \frac{1-\|\nabla_\hor r_\varepsilon\|^2}{r_\varepsilon}.
 \end{equation}
\end{theorem}
\begin{proof}
Let $\gamma : [0,r_\varepsilon] \to \M$ be the unique $g_\varepsilon$-geodesic, parametrized with unit speed, joining $x$ with $y$. Recall that $\dot\gamma = \nabla_\hor r_\varepsilon + \varepsilon \nabla_\ver r_\varepsilon$. Let
\begin{equation}
W_0(t) = \frac{t}{r_\varepsilon} \dot\gamma_{\hor,\perp}(t), \qquad\text{where}\qquad \dot\gamma_{\hor,\perp} =\dot\gamma_\hor(t) - h^2 \dot\gamma(t), \qquad \forall t\in [0,r_\varepsilon].
\end{equation}
Notice that $W_0$ is $g_\varepsilon$-orthogonal to $\dot\gamma$ by construction, and \emph{not} horizontal. Furthermore $\hat{\nabla}^\varepsilon_{\dot\gamma} W_0 = \tfrac{1}{r_\varepsilon} \dot\gamma_{\hor,\perp}$. Since $\hat{\nabla}^\varepsilon$ is metric, one immediately deduces that for a general foliation
\begin{equation}
\langle\hat{\nabla}^\varepsilon_{\dot\gamma}\nabla^\varepsilon_{\dot\gamma} W_0 + \hat{R}^\varepsilon(W_0,\dot\gamma)\dot\gamma,W_0\rangle_\varepsilon =0,
\end{equation}
so that condition \eqref{eq:comparison-condition} of Theorem \ref{t:comparison} is verified. We obtain hence from the latter:
\begin{equation}\label{eq:comparisongeod2}
\mathrm{Hess}^{\hat{\nabla}^\varepsilon}(r_\varepsilon)(\dot\gamma_{\hor,\perp},\dot\gamma_{\hor,\perp}) \leq \frac{\|\dot\gamma_{\hor,\perp}\|^2_\varepsilon}{r_\varepsilon} = \frac{h^2(1-h^2)}{r_\varepsilon}.
\end{equation}
Using respectively the eikonal equation $\|\nabla^{g_\varepsilon} r_\varepsilon\|_\varepsilon =1$, and the fact that $\hat{\nabla}^\varepsilon$ has completely skew-symmetric torsion (cf.\ Lemma \ref{l:constructionskew}), one has
\begin{equation}
\mathrm{Hess}^{\hat{\nabla}^\varepsilon}(r_\varepsilon)(\dot\gamma_{\hor,\perp},\dot\gamma_{\hor,\perp}) = \mathrm{Hess}^{\hat{\nabla}^\varepsilon}(r_\varepsilon)(\dot\gamma_{\hor},\dot\gamma_{\hor}) = \mathrm{Hess}(r_\varepsilon)(\dot\gamma_{\hor},\dot\gamma_{\hor}).
\end{equation}
This fact, together with \eqref{eq:comparisongeod2}, completes the proof.
To relate the Hessian in \eqref{eq:comparisongeod2} with Bott's one observe that, for any $g_\varepsilon$-metric connection $D$, with torsion $\Tor$ it holds
\begin{equation}\label{eq:hesssymm}
\mathrm{Hess}^D(u)(X,Y) = \mathrm{Hess}^D(u)(Y,X) - \langle \nabla u, \Tor(X,Y)\rangle_\varepsilon, \qquad \forall X,Y \in T\M.
\end{equation}
Since $r_\varepsilon$ satisfies the eikonal equation $\|\nabla^{g_\varepsilon} r_\varepsilon\|_\varepsilon =1$, it holds
\begin{equation}
\mathrm{Hess}^D(r_\varepsilon)(X,Y+\nabla^{g_\varepsilon} r_\varepsilon) = \mathrm{Hess}^D(r_\varepsilon)(X,Y), \qquad \forall X,Y \in T\M.
\end{equation}
Combining this with \eqref{eq:hesssymm}, we obtain
\begin{equation}
\mathrm{Hess}^D(r_\varepsilon)(X+\nabla^{g_\varepsilon} r_\varepsilon,Y) = -\langle \nabla^{g_\varepsilon} r_\varepsilon, \Tor(\nabla^{g_\varepsilon} r_\varepsilon,Y)\rangle_\varepsilon, \qquad \forall X,Y \in T\M.
\end{equation}
Set $D=\hat{\nabla}^\varepsilon$. In this case $\Tor$ is completely skew-symmetric (cf.\ Lemma \ref{l:constructionskew}), and hence
\begin{equation}\label{eq:nullspace}
\mathrm{Hess}^{\hat{\nabla}^\varepsilon}(r_\varepsilon)(X+\nabla^{g_\varepsilon} r_\varepsilon, Y+ \nabla^{g_\varepsilon} r_\varepsilon) = \mathrm{Hess}^{\hat{\nabla}^\varepsilon}(r_\varepsilon)(X, Y), \qquad \forall X,Y\in T\M.
\end{equation}
If $X=Y=\nabla_\hor r_\varepsilon$, the above Hessians can be computed equivalently with respect to $\nabla$ (cf.\ Remark \ref{rmk:hessian}). The result follows.
\end{proof}

In the next two sections we discuss Hessian comparison theorems in the remaining directions of the splitting introduced in Section \ref{s:splitting}.

\subsection{Comparison of Riemannian type} \label{sec:ComparisonRiemannian}

Let $x\in \M$ be fixed and $y\notin \Cut_\ve(x)$, so that $r_\ve = d_\ve(x,\cdot)$ is smooth in a neighbourhood of $y$. In this case, provided that $\nabla_\hor r_\ve \neq 0$, we set for brevity
\begin{equation} \label{notationsubspace}
\hor_\rie(y) := \hor_\rie(\nabla r_\ve(y)), \qquad \hor_\sas(y) :=\hor_\sas(\nabla r_\ve(y)).
\end{equation}
In this section we focus on the index form in the space $\hor_{\rie}$. This will yield comparison theorems of Riemannian type. Recall that, for an horizontal vector $X$, the condition $X\in \hor_\rie(Y)$ is equivalent to the condition $X\perp Y$ and $X \perp J_Z Y$ for all $Z\in \ver$.

\begin{theorem}\label{t:Riemannian-sect}
Let $(M,\hor,g)$ be an H-type foliation with parallel horizontal Clifford structure, satisfying the $J^2$ condition. Assume that there exists $\rho \in \R$ such that
\begin{equation}
\mathrm{Sec}(X \wedge Y) \geq \rho,\qquad \forall X,Y\in \hor, \quad \text{with} \quad X \in \hor_\rie(Y).
\end{equation}
Let $y\notin \Cut_\ve(x)$ with $\nabla_\hor r_\ve(y) \neq 0$. Let $X\in \hor_{\rie}(y)$, with $\|X\|=1$. Then at $y$ it holds
\begin{equation}
\mathrm{Hess}(r_\varepsilon)(X,X) \leq  F_{\mathrm{Riem}}(r_\varepsilon,K), \qquad K = \rho \|\nabla_\hor r_\varepsilon\|^2 + \tfrac{1}{4}\|\nabla_\ver r_\varepsilon\|^2.
\end{equation}
Provided that $K>0$, it follows that
\begin{equation}
r_\varepsilon(y) < \frac{\pi}{\sqrt{\rho \|\nabla_\hor r_\varepsilon\|^2 + \tfrac{1}{4}\|\nabla_\ver r_\varepsilon\|^2}}.
\end{equation}
\end{theorem}

\begin{theorem}\label{t:Riemannian-average}
Let $(\M,\hor,g)$ be an H-type foliation with parallel horizontal Clifford structure, satisfying the $J^2$ condition. Assume that $n-m-1>0$ and that there exists $\rho \in \R$ such that
\begin{equation}
\sum_{i=1}^{n-m-1} \mathrm{Sec}(X \wedge X_i) \geq  (n-m-1)\rho,\qquad \forall X\in \hor,
\end{equation}
where $X_1,\dots,X_{n-m-1}$ is an orthonormal frame for $\hor_{\rie}(X)$. Then, for all $y\notin \Cut_\ve(x)$ with $\nabla_\hor r_\ve(y) \neq 0$ it holds
\begin{equation}
\sum_{i=1}^{n-m-1} \mathrm{Hess}(r_\varepsilon)(X_i,X_i) \leq (n-m-1) F_{\mathrm{Riem}}(r_\varepsilon,K),  \qquad K = \rho \|\nabla_\hor r_\varepsilon\|^2 +\tfrac{1}{4}\|\nabla_\ver r_\varepsilon\|^2.
\end{equation}
where here $X_1,\dots,X_{n-m-1}$ is an orthonormal basis for $\hor_{\rie}(y)$. Provided that $K >0$, it follows that
\begin{equation}
r_\varepsilon(y) < \frac{\pi}{\sqrt{\rho \|\nabla_\hor r_\varepsilon\|^2 + \tfrac{1}{4}\|\nabla_\ver r_\varepsilon\|^2}}.
\end{equation}
\end{theorem}

\begin{proof}[Proof of Theorems \ref{t:Riemannian-sect} and \ref{t:Riemannian-average}]
For both theorems, fix $\varepsilon>0$, $x \in \M$ and $y \notin \Cut_\varepsilon(x)$. Let $\gamma : [0,r_\varepsilon] \to \M$ be the unique $g_\varepsilon$-geodesic, parametrized with unit speed, joining~$x$ with~$y$. Observe that along $\gamma$, we have the identity $\hor_\rie(\nabla r_\varepsilon) = \hor_{\rie}(\nabla^{g_\varepsilon} r_\varepsilon) = \hor_\rie(\dot\gamma)$.

Write $h = \| \nabla_\ch r_\ve\|$ and $v = \| \nabla_\V r_\ve\|$. By Lemma \ref{l:Jacobiequation}, the Jacobi equation for a vector field $W\in\hor_\rie(\dot\gamma)$ reads
\begin{equation}
\hat{\nabla}^\varepsilon_{\dot\gamma}\hat{\nabla}^\varepsilon_{\dot\gamma} W - J^\varepsilon_{\dot\gamma}\hat{\nabla}^\varepsilon_{\dot\gamma} W+R_\hor(W,\dot\gamma)\dot\gamma = 0.
\end{equation}
Let $X\in\hor_\rie(\dot\gamma)$ be a $\hat{\nabla}^\varepsilon$-parallel and unit vector field along $\gamma$ which is possible by Proposition \ref{p:parallelism}. Using the horizontally parallel torsion condition, we have identity $(\hat{\nabla}^\varepsilon_{\dot\gamma}J)_{\dot\gamma} = 0$ (cf.\ Lemma \ref{l:someproperties}). It follows that $Y=\tfrac{1}{\varepsilon v}J_{\dot\gamma}X \in\hor_\rie(\dot\gamma)$, is also a $\hat{\nabla}^\varepsilon$-parallel and unit vector field along $\gamma$, perpendicular to $X$. Furthermore one has
\begin{equation}
J_{\dot\gamma} X =  \varepsilon  v Y, \qquad J_{\dot\gamma} Y= - \varepsilon v X.
\end{equation}
Set hence
\begin{equation}
W(t) = f(t) X(t) + g(t) Y(t), \qquad \forall t \in [0,r_\varepsilon],
\end{equation}
for some functions $f,g$ to be chosen later, and such that $f(0)=f(r_\varepsilon)-1 = g(0)=g(r_\varepsilon)=0$. If the horizontal curvature is constant and equal to $\rho$, the Jacobi equation would be
\begin{align*}
\ddot{f} + \dot{g} v  + f \rho h^2 =0, \\
\ddot{g} - \dot{f} v  + g \rho h^2 =0.
\end{align*}
This system has explicit solution, for the given boundary conditions, provided that $r_\varepsilon < \pi/ \sqrt{K}$ (with the convention $1/\sqrt{K} = +\infty$ if $K\leq 0$). It can be most easily expressed in terms of $f(t) + i g(t)$, and it is given by
\begin{equation}
f(t) + i g(t) = e^{i\tfrac{ v (t-r)}{2}} \frac{\sin( t \sqrt{K})}{\sin(r_\varepsilon \sqrt{K})}, \qquad K = \rho h^2  + \frac{v^2}{4},
\end{equation}
where the above expression is considered as an analytic function of $K \in \R$. We want to apply Theorem \ref{t:comparison} to $W$. Using Lemma \ref{l:Jacobiequation}, and the curvature lower bound, one has
\begin{align*}
\langle \hat{\nabla}^\varepsilon_{\dot\gamma}\nabla^{\varepsilon}_{\dot\gamma} W + \hat{R}_\varepsilon(W,\dot\gamma)\dot\gamma,W\rangle & = f(\ddot{f} + \dot{g} v)+ g(\ddot{g} -\dot{f} v) + R_\hor(W,\dot\gamma,\dot\gamma,W) \\
& \geq f(\ddot{f} + \dot{g} v +f \rho h^2)+ g(\ddot{g} -\dot{f} v + g \rho h^2) =0 .
\end{align*}
We can hence apply Theorem \ref{t:comparison} (cf.\ also Remark \ref{rmk:hessian}), obtaining
\begin{equation}
\mathrm{Hess}(r_\varepsilon)(X,X) \leq \dot{f}(r_\varepsilon) = F_{\rie}(r_\varepsilon, K).
\end{equation}

We now prove Theorem \ref{t:Riemannian-average}. Choose a $\hat{\nabla}^\varepsilon$-parallel frame $X_1,\dots,X_{n-m-1} \in \hor_{\rie}(\dot\gamma)$ along $\gamma$, thanks to Proposition \ref{p:parallelism}. Set $Y_i = \tfrac{1}{\varepsilon v} J_{\dot\gamma} X_i$, as in the proof above, and $W_i = f(t) X_i(t) + g(t) Y_i(t)$ for all $i=1,\dots,n-m-1$. Choosing $f$ and $g$ as in the first part of the proof, and using the Ricci-type curvature assumption, one has
\begin{equation}
\sum_{i=1}^{n-m-1} \langle \hat{\nabla}^\varepsilon_{\dot\gamma}\nabla^\varepsilon_{\dot\gamma} W_i + \hat{R}^\varepsilon(W,\dot\gamma)\dot\gamma,W\rangle \geq 0.
\end{equation}
Hence we can apply Theorem \ref{t:comparison} to the $(n-m-1)$-tuple $W_i$, obtaining the statement.
\end{proof}

\subsection{Comparison of Sasakian type} \label{sec:ComparisonSasakian}

If in this section we focus on $\hor_\sas$, in the general setting of H-type foliations with horizontally parallel torsion. The $J^2$ condition is not necessary. The first case is very particular, as it does not require nor the $J^2$ condition nor the parallel horizontal Clifford structure assumption. If $m=1$, this case exhaust the whole set of Sasakian directions and we obtain an extension of the results proved in \cite{BGKT17}.

\begin{theorem}\label{t:SasakianZpar}
Let $(\M,\hor,g)$ be an H-type foliation with horizontally parallel torsion. Assume that there exists $\rho \in \R$ such that
\begin{equation}
\mathrm{Sec}(X \wedge J_Z X) \geq \rho,\qquad \forall X\in \hor,\, Z\in \ver.
\end{equation}
Let $y \notin \Cut_\ve(x)$ with $\nabla_\hor r_\ve(y) \neq 0$. Let $Z \in \ver$ be non-zero, with $Z\propto \nabla_\ver r_\ve(y)$ (if $\nabla_\ver r_\ve(y) = 0$ the latter condition imposes no restriction on $Z$), and set
\begin{equation}
X = \frac{J_{Z}\nabla_\hor r_\varepsilon}{ \|Z\|\|\nabla_\hor r_\varepsilon\| }.
\end{equation}
Then at $y$ it holds
\begin{equation}
\mathrm{Hess}(r_\varepsilon)(X,X) \leq  F_{\mathrm{Sas}}(r_\varepsilon, K), \qquad K = \rho \|\nabla_\hor r_\varepsilon\|^2 + \|\nabla_\ver r_\varepsilon\|^2.
\end{equation}
Provided that $K >0$, it follows that
\begin{equation}
r_\varepsilon(y) < \frac{2\pi}{\sqrt{\rho \|\nabla_\hor r_\varepsilon\|^2 + \|\nabla_\ver r_\varepsilon\|^2}}.
\end{equation}
\end{theorem}
\begin{proof}
Fix $\varepsilon>0$. Let $\gamma : [0,r_\varepsilon] \to \M$ be the unique $g_\varepsilon$-geodesic, parametrized with unit speed, joining $x$ with $y$. Observe that $\hor_\rie(\nabla r_\varepsilon) = \hor_{\rie}(\nabla^{g_\varepsilon} r_\varepsilon) = \hor_\rie(\dot\gamma)$.

Let $Z \in \Gamma(\ver)$ be a $\hat{\nabla}^\varepsilon$-parallel vector field along $\gamma$ and normalized in such a way that $\|J_Z\dot\gamma\|=1$. Set hence
\begin{equation}
W(t) =a(t) J_Z \dot\gamma + b(t) Z_\perp, \qquad \forall t \in [0,r_\varepsilon],
\end{equation}
for some functions $a,b$ to be chosen later, with $a(0)=b(0)=0$ and $a(r_\varepsilon)-1= b(r_\varepsilon) =0$, and where we have set
\begin{equation}
Z_\perp  = Z - \langle Z,\dot\gamma\rangle_\varepsilon \dot\gamma,
\end{equation}
in such a way that $W$ is $g_\varepsilon$-orthogonal to $\dot\gamma$. If $v=\|\nabla_\ver r_\ve \| > 0$, we choose $Z=\tfrac{1}{\varepsilon v h}\dot\gamma_\ver$. If $v=0$, we can choose any $\hat\nabla^\ve$-parallel and vertical $Z$, normalized in such a way that $\|J_Z\dot\gamma\|=1$. In both cases, by Lemma \ref{l:Jacobiequation}, one has
\begin{equation}
\hat{\nabla}^\varepsilon_{\dot\gamma}\nabla^\varepsilon_{\dot\gamma} W + \hat{R}^\varepsilon(W,\dot\gamma)\dot\gamma = \frac{1}{\varepsilon}\left(\ddot{a} \varepsilon + \dot{b} -a +a \varepsilon v^2\right) J_Z\dot\gamma + a R_\hor(J_Z\dot\gamma,\dot\gamma)\dot\gamma + (\ddot{b}-\dot{a})Z_\perp.
\end{equation}
In the case of constant sectional curvature equal to $\rho$, $W(t)$ is a Jacobi field if and only if
\begin{align*}
\ddot{a} \varepsilon + \dot{b} -a +a \varepsilon K =0,\\
\ddot{b}-\dot{a}=0.
\end{align*}
where $K = \rho h^2 + v^2$. Assume first that $K\neq 0$. The general solution is found via elementary methods, and it is
\begin{align*}
a(t) & = C_1 \cos(\sqrt{K} t) + C_2\sin(\sqrt{K} t) + C_3,\\
b(t) & = C_1 \frac{\sin(\sqrt{K} t)}{\sqrt{K}} + C_2\frac{1-\cos(\sqrt{K}t)}{\sqrt{K}}  +C_3(1-K\varepsilon)t + C_4,
\end{align*}
where the above expressions are considered as analytic functions of $K$. The constants $C_1,C_2,C_3,C_4$ are uniquely  determined by the boundary conditions $a(0)=b(0)=a(r_\varepsilon)-1=b(r_\varepsilon)=0$, provided that $r_\varepsilon < 2\pi/ \sqrt{K}$ (with the convention $1/\sqrt{K} = +\infty$ if $K\leq 0$).

Using Lemma \ref{l:Jacobiequation}, the curvature lower bound, and the fact that $a(t),b(t)$ solve the Jacobi equation in the constant curvature case, one obtains
\begin{align*}
\langle \hat{\nabla}^\varepsilon_{\dot\gamma}\nabla^{\varepsilon}_{\dot\gamma} W + \hat{R}_\varepsilon(W,\dot\gamma)\dot\gamma,W\rangle_\varepsilon \geq 0.
\end{align*}
We can hence apply Theorem \ref{t:comparison} (cf.\ also Remark \ref{rmk:hessian}), obtaining
\begin{equation}\label{eq:sasakianparallalast}
\mathrm{Hess}(r_\varepsilon)(J_Z\dot\gamma,J_Z\dot\gamma)\leq \frac{\sqrt{K}(\sin (r_\varepsilon \sqrt{K})- (1-\varepsilon K)r_\varepsilon \sqrt{K} \cos (r_\varepsilon \sqrt{K}))}{ 2-2\cos(r_\varepsilon \sqrt{K})-(1-\varepsilon K) r_\varepsilon \sqrt{K} \sin (r_\varepsilon \sqrt{K})}.
\end{equation}

An important observation is that the derivative with respect to the parameter $\varepsilon$ (and fixed value of $r_\varepsilon)$ of the r.h.s. of \eqref{eq:sasakianparallalast} is always negative. Henceforth we can upper bound it with its value at $\varepsilon=0$, yielding the Sasakian function $F_{\mathrm{Sas}}(r_\varepsilon,K)$.

This concludes the proof for $K\neq 0$. The case $K=0$ is obtained similarly.
\end{proof}

\begin{remark}\label{vertical gradient comparison}
We can read the previous result as a bound for $\|\nabla_\ver r_\varepsilon\|$. Provided that $\rho \geq 0$, we get for all $y\notin \Cut_\ve(x)$ and with $\nabla_\hor r_\ve(y) \neq 0$. 
\begin{equation}
r_\varepsilon \|\nabla_\ver r_\varepsilon\| \leq 2\pi.
\end{equation}
Taking the limit for $\varepsilon \to 0$ (and using Proposition \ref{p:cutapprox}), this yields a cotangent injectivity radius bound for the sub-Riemannian exponential map, which is sharply attained in the case of the Heisenberg groups.
\end{remark}

\begin{theorem}\label{t:SasakianZperp}
Let $(\M,\hor,g)$ be an H-type foliation with parallel horizontal Clifford structure, and satisfying the $J^2$ condition. Let $\kappa^2$ be the curvature of the vertical leaves. Assume that there exists $\rho \in \R$ such that
\begin{equation}
\mathrm{Sec}(X\wedge J_ZX) \geq \rho,\qquad \forall X\in \hor,\, Z\in \ver.
\end{equation}
Let $y \notin \Cut_\ve(x)$, with $\nabla_\hor r_\ve(y) \neq 0$. Let $Z \in \ver$ be non-zero with $Z\perp \nabla_\ver r_\varepsilon(y)$. Set
\begin{equation}
X = \frac{J_Z \nabla_\hor r_\varepsilon}{\|Z\|\| \nabla_\hor r_\varepsilon\|}.
\end{equation}
Then, at $y$, it holds
\begin{equation}
\mathrm{Hess}(r_\varepsilon)(X,X) \leq  F_{\mathrm{Sas}}(r_\varepsilon,K_2),
\end{equation}
where
\begin{equation}
K_2 = \rho \|\nabla_\hor r_\varepsilon\|^2 +(2-\kappa\varepsilon)(\kappa\varepsilon-1)\|\nabla_\ver r_\varepsilon\|^2.
\end{equation}
Provided that $K_2>0$, it follows that
\begin{equation}
r_\varepsilon(y) < \frac{2\pi}{\sqrt{\rho \|\nabla_\hor r_\varepsilon\|^2 +(2-\kappa\varepsilon)(\kappa\varepsilon-1)\|\nabla_\ver r_\varepsilon\|^2}}.
\end{equation}
\end{theorem}
\begin{proof}
Fix $\varepsilon>0$. Let $\gamma : [0,r_\varepsilon] \to \M$ be the unique $g_\varepsilon$-geodesic, parametrized with unit speed, joining $x$ with $y$.

Let $Z\in \Gamma(\ver)$ be a $\hat{\nabla}^\varepsilon$-parallel vector field along $\gamma$, with $Z \perp \dot\gamma_\ver$, with $\|Z\|=1$. Assume first that $v=\|\nabla_\ver r_\ve\| \neq 0$ along $\gamma$. Consider the tuple of $g$-orthonormal vectors
\begin{equation}
X = \frac{J_Z\dot\gamma}{h}, \qquad Y = \frac{J_{\dot\gamma}J_Z\dot\gamma}{h\varepsilon v}, \qquad Z, \qquad V= \frac{T(\dot\gamma,J_{\dot\gamma}J_Z\dot\gamma)}{h^2\varepsilon v}.
\end{equation}
Set therefore
\begin{equation}\label{eq:abcd}
W(t)= a(t) X  + c(t) Y + b(t) Z+ d(t) V, \qquad \forall t\in [0,r_\varepsilon].
\end{equation}
The parallel horizontal Clifford structure yields (cf.\ Lemma \ref{l:someproperties}) 
\begin{align*}
\hat{\nabla}^\varepsilon_{\dot\gamma} X & = (2-\kappa\varepsilon)v Y,   & J_{\dot\gamma} X & = \varepsilon v Y, \\
\hat{\nabla}^\varepsilon_{\dot\gamma}Y  & = -(2-\kappa\varepsilon)v X,  & J_{\dot\gamma} Y & = -\varepsilon v X, \\
\hat{\nabla}^\varepsilon_{\dot\gamma}Z& = \hat{\nabla}^\varepsilon_{\dot\gamma}V = 0.
\end{align*}
Notice also that
\begin{equation}
T(\dot\gamma,X) =  h Z, \qquad T(\dot\gamma,Y) =  h V.
\end{equation}
In the case of constant curvature equal to $\rho$, $W(t)$ is a Jacobi field if and only if
\begin{align*}
\ddot{a} + v(2\kappa\varepsilon-3) \dot{c} + \tfrac{h}{\varepsilon}(\dot{b}-ha) + \left[ \rho h^2+(2-\kappa\varepsilon)(\kappa\varepsilon-1)v^2\right]a - \kappa h v d =0,\\
\ddot{c} - v(2\kappa\varepsilon-3) \dot{a} + \tfrac{h}{\varepsilon}(\dot{d}-h c) + \left[\rho h^2+(2-\kappa\varepsilon)(\kappa\varepsilon-1)v^2\right]c + \kappa h v b =0,\\
\ddot{b} - h \dot{a} + \kappa \varepsilon v h c + \kappa^2 \varepsilon^2 v^2 b = 0,\\
\ddot{d} - h \dot{c}- \kappa \varepsilon v h a + \kappa^2 \varepsilon^2 v^2 d = 0.
\end{align*}
Set from now
\begin{equation}
K_2  = \rho h^2+(2-\kappa\varepsilon)(\kappa\varepsilon-1)v^2.
\end{equation}
We will drop the explicit dependence on $\varepsilon$ in the following of the proof. In terms of the complex variables $\theta(t) = a(t) + i c(t)$ and $\phi(t) = b(t) + i d(t)$, the above system reads
\begin{align}
\ddot{\theta} - i v (2\kappa\varepsilon-3)\dot{\theta} + \frac{h}{\varepsilon}(\dot{\phi}-h \theta) + K \theta + i \kappa h v \phi = 0,\label{eq:completesystem1} \\
\ddot{\phi} - h \dot{\theta} - i \kappa  \varepsilon v h \theta + \kappa^2 \varepsilon^2 v^2 \phi =0. \label{eq:completesystem2}
\end{align}
We consider the simplified version of the system \eqref{eq:completesystem1}-\eqref{eq:completesystem2} obtained by suppressing all the terms with purely imaginary coefficients:\begin{align}
\ddot{\theta} + \frac{h}{\varepsilon}(\dot{\phi}-h \theta) + K \theta  = 0, \label{eq:syssymp1}\\
\ddot{\phi} - h \dot{\theta}  =0. \label{eq:syssymp2}
\end{align}
In this case, provided we can find a solution, we observe that
\begin{equation}\label{eq:remainder}
\langle \hat{\nabla}^\varepsilon_{\dot\gamma}\nabla^\varepsilon_{\dot\gamma} W + \hat{R}^\varepsilon(W,\dot\gamma)\dot\gamma,W\rangle_\varepsilon= 2\kappa h v \, \mathrm{Im}(\theta\bar{\phi}) +  \varepsilon\kappa^2 v^2 \phi,
\end{equation}
where $\mathrm{Im}$ denotes the imaginary part. Since the system \eqref{eq:syssymp1}-\eqref{eq:syssymp2} has real coefficients, a solution respecting the boundary conditions $\theta(0)=\phi(0)=\theta(r_\varepsilon)-1=\phi(r_\varepsilon)=0$ is real. Hence, the r.h.s.\ of \eqref{eq:remainder} is non-negative, and thus condition \eqref{eq:comparison-condition} is satisfied. The simplified system has the same form of the one for the case $Z \propto \dot\gamma_\ver$, with minor differences, and can be solved with the same computation as in the proof of Theorem \ref{t:SasakianZpar}. We obtain
\begin{align*}
\mathrm{Hess}(r_\varepsilon)\left(\frac{J_Z\dot\gamma}{h},\frac{J_Z\dot\gamma}{h}\right) & \leq \frac{\sqrt{K}(\sin (r_\varepsilon \sqrt{K})- (1-\tfrac{\varepsilon K}{h^2})r_\varepsilon \sqrt{K} \cos (r_\varepsilon \sqrt{K}))}{ 2-2\cos(r_\varepsilon \sqrt{K})-(1-\tfrac{\varepsilon K}{h^2}) r_\varepsilon \sqrt{K} \sin (r_\varepsilon \sqrt{K})}  \leq F_{\mathrm{Sas}}(r_\varepsilon,K_2),
\end{align*}
where we used the monotonicity with respect to  $\varepsilon$ (cf.\ proof of Theorem \ref{t:SasakianZpar}).

The case left out, that is $v =\|\nabla_\ver r_\ve\|=0$ is simpler. In this case proceeds exactly in the same way, by formally setting $c(t)= d(t)= 0$ in \eqref{eq:abcd} and all the subsequent computations. We omit the details.
\end{proof}

\begin{remark}
This remark is the vertical counterpart of Remark \ref{rmk:hessian}. If $(\M,\hor,g)$ is a totally geodesic foliation, $W\in \Gamma(\ver)$ and $u$ is sufficiently regular,  then
\begin{multline}
\mathrm{Hess}^{\hat{\nabla}^\varepsilon}(u)(W,W) =\langle \hat{\nabla}^\varepsilon_W \nabla^{g_\ve} u ,W\rangle_\varepsilon = \langle (\nabla_W +J^\ve_W)\nabla^{g_\ve} u,W\rangle_\varepsilon\\
 = \langle \nabla_W \nabla u,W\rangle =\mathrm{Hess}(u)(W,W).
\end{multline}
If in Theorem \ref{t:SasakianZpar} we instead solve the system with initial conditions $a(r_\varepsilon) = b(r_\varepsilon) - 1 = 0$ we can use the above to recover a bound on the Hessian in vertical directions parallel to the geodesic, specifically for 
\begin{equation} Z \propto \nabla_\ver r_\varepsilon, \quad \|Z\| = \frac{1}{\|\nabla_\hor r_\varepsilon\|} \end{equation}
it holds that at $y \notin \Cut_\ve(x)$,
\begin{equation}
\mathrm{Hess}(r_\varepsilon)(Z,Z) \leq  \frac{12}{r_\varepsilon^3}.
\end{equation}

Similarly, if in the Theorem \ref{t:SasakianZperp} we instead solve the system with initial conditions $\theta(r_\varepsilon) = \phi(r_\varepsilon) - 1 = 0$ we can recover a bound on the Hessian in vertical directions perpendicular to the geodesic, specifically for 
\begin{equation} Z \perp \nabla_\ver r_\varepsilon(y), \quad \|Z\| = \|\nabla_\hor r_\varepsilon \| \end{equation}
it holds that at $y \notin \Cut_\ve(x)$,
\begin{equation}
\mathrm{Hess}(r_\varepsilon)(Z,Z) \leq  \frac{12}{r_\varepsilon^3}.
\end{equation}
\end{remark}

\subsection{Sub-Riemannian diameter bounds} 

We will summarize the results of Section~\ref{sec:ComparisonRiemannian} and \ref{sec:ComparisonSasakian} relating curvature and diameter into one statement. We introduce the following terminology. Recall that $\hor$ and $\V$ have respective ranks $n$ and $m$. Define a symmetric two-tensor $\Ric_{\hor}$ such that if $R$ is the curvature of the Bott connection $\nabla$ then
\begin{equation}
\Ric_\hor(X,X)= \sum_{i=1}^n \langle R(W_i, X) X, W_i \rangle,
\end{equation}
for any $X \in T \M$ with $W_1, \dots, W_n$ being a $g$-orthonormal basis of $\hor$. Decompose this tensor as $\Ric_\hor = \Ric_{\rie} + \Ric_{\sas}$, where for any $X \in T\M \setminus \V$,
\begin{equation}
\Ric_{\sas}(X,X) = \frac{1}{\| \pi_\hor X\|^2} \sum_{i=1}^m \langle R(J_{Z_i} X, \pi_\hor X) \pi_\hor X, J_{Z_i} X \rangle,
\end{equation}
where $Z_1, \dots, Z_m$ is a $g$-orthonormal basis of $\V$. In other words, $\Ric_{\rie}$ and $\Ric_{\sas}$ represent the parts of the horizontal Ricci curvature $\Ric_\hor(X,X)$ obtained by tracing over $\hor_{\rie}(X)$ and $\hor_{\sas}(X)$, respectively. We note that for $m=1$, in the context of Sasakian manifolds, $\Ric_{\sas}$ coincides with the pseudo-Hermitian sectional curvature introduced and studied in \cite{Barletta}. 

\begin{theorem}
Assume that $(\M,g)$ is complete. Write $\diam_0 (\M)$ for the diameter with respect to the sub-Riemannian distance. Let $\rho >0$ be a positive number.
\begin{enumerate}[\rm (a)]
\item Assume that $(\M,\hor,g)$ is an H-type foliation with parallel horizontal Clifford structure, satisfying the $J^2$ condition. Assume $n-m-1>0$ and that for any $X \in \hor$, we have
\begin{equation}
\Ric_{\rie}(X,X) \geq (n-m-1)\rho \| X\|^2.
\end{equation}
Then $\M$ is compact, with finite fundamental group and
\begin{equation}
\diam_0( \M) \leq \frac{\pi}{\sqrt{\rho}}.
\end{equation}
\item Assume that $(\M,\hor,g)$ is an H-type foliation with horizontally parallel torsion. Assume that for any unit $X \in \hor$ and unit $Z \in \V$, we have
\begin{equation}
\mathrm{Sec}(X \wedge J_Z X) \geq \rho.
\end{equation}
Then $\M$ is compact, with finite fundamental group and
\begin{equation}
\diam_0( \M) \leq \frac{2\pi}{\sqrt{\rho}}.
\end{equation}
\item Assume that $(\M,\hor,g)$ is an H-type foliation with parallel horizontal Clifford structure, and satisfying the $J^2$ condition. Assume that for any $X \in \hor$, we have
\begin{equation}
 \Ric_{\sas}(X,X) \geq m \rho \| X\|^2.
\end{equation}
Assume also that at least one of the following conditions are satisfied
\begin{enumerate}[\rm (i)]
\item For any unit $X \in \hor$ and unit $Z \in \V$, we have $\mathrm{Sec}(X \wedge J_Z X) \geq 0$,
\item $n-m-1 >0$ and for any $X \in \hor$, $\Ric_{\rie}(X) \geq 0$.
\end{enumerate}
Then $\M$ is compact, with finite fundamental group and
\begin{equation}
\diam_0(\M)\leq \frac{2\pi \sqrt{3}}{\sqrt{\rho}}.
\end{equation}
\end{enumerate}
\end{theorem}
We point out that points (a) and (b) are sharp, as the upper diameter bound is attained for the case of the standard sub-Riemannian structure on the complex, quaternionic or octonionic Hopf fibrations (see e.g.\ \cite{BW1,BW2,BG-octonionic}, respectively).
\begin{proof}
For the proof, we will go to the universal cover $\widetilde{\M}$ of $\M$ and prove that this is of bounded diameter. The result of (a) and (b) follows by taking the limit $\ve\to 0$ of respectively Theorem~\ref{t:Riemannian-average} and Theorem~\ref{t:SasakianZpar}. For the result in (c), pick any pair of points $x,y \in \widetilde{\M}$ which are not in $\Cut_0(\widetilde{\M})$ and satisfy
\begin{equation}
d_0(x,y) > \frac{2\pi \sqrt{2}}{\sqrt{\rho}}.
\end{equation}
If no such points exist, then $\diam_0 (\widetilde{\M}) \leq \frac{2\pi \sqrt{2}}{\sqrt{\rho}}$, and the statement holds. If such points $(x,y)$ exist, write $r_\ve(\cdot) = d_\ve(x,\cdot)$. We note by Proposition~\ref{p:cutapprox} we have that for some $\ve'' > 0$, we will also have $(x,y) \not \in \mathbf{Cut}_\ve(\M)$ and $r_\ve(y) > \frac{2\pi \sqrt{2}}{\sqrt{\rho}}$ for any $0 \leq \ve \leq \ve''$. We must also have $\| \nabla_\V r_0 \|(y) >0$ in this case, as Theorem~~\ref{t:SasakianZperp} would otherwise imply that $r_0(y) < \frac{2\pi}{\sqrt{\rho}}$. 
If $\| \nabla_\V r_0 \|(y) > 0$, then either (i) or (ii) implies as $\ve \to 0$
\begin{equation}
r_0(y) \| \nabla_\V r_0 \|(y) \leq 2 \pi,
\end{equation}
by respectively Theorem~\ref{t:Riemannian-average} and Theorem~\ref{t:SasakianZpar}. In particular,
\begin{equation}
\rho - 2 \| \nabla_\V r_0 \|^2(y)  \geq \rho - \frac{8 \pi^2}{r_0(y)^2}> \rho - \rho = 0.
\end{equation}
Hence, it follows from Theorem~\ref{t:SasakianZperp} that
\begin{equation}
r_0(y) \leq \frac{2 \pi}{\sqrt{\rho - \frac{8 \pi^2}{r_0(y)^2 }}}.
\end{equation}
Solving this inequality yields $r_0^2(y) \rho  \leq 12 \pi^2$, and the resulting diameter bound follows.
\end{proof}


\subsection{Sub-Laplacian comparison theorems}

Let $(\M,\hor,g)$ be a totally geodesic foliation. The sub-Laplacian $\Delta_\hor$ is defined as the second order differential operator associated with the Dirichlet form
\begin{equation}
\mathcal{E}(u,v) = \int_{\M} \langle \nabla_\hor u,\nabla_\hor v\rangle d\mu_g, \qquad u,v \in C^\infty_c(\M),
\end{equation}
where $d\mu_g$ denotes the Riemannian measure. Since $\hor$ is completely non-holonomic, the horizontal Laplacian is hypoelliptic \cite{Hor67}, and if the sub-Riemannian metric is complete then $\Delta_\hor$ is essentially self-adjoint on the space of smooth and compactly supported functions, as proven in \cite{Strichartz}.

One can check that, for a totally geodesic foliation, the sub-Laplacian coincides with the horizontal trace of the Hessian with respect to the Bott connection:
\begin{equation}
\Delta_\hor u = \sum_{i=1}^{n} \mathrm{Hess}(u)(X_i,X_i).
\end{equation}
For structures (not necessarily foliations) respecting the H-type condition, one can check that the Riemannian measure $\mu_g$ is proportional to the intrinsic Popp's measure of the sub-Riemannian structure, and thus $\Delta_\hor$ coincides with the intrinsic sub-Laplacian defined with respect to Popp's measure, cf.\ \cite{BR-Popp}. 

For any $Y\in \Gamma(\hor)$ recall the orthogonal decomposition of the horizontal bundle (cf.\ Section \ref{s:splitting})
\begin{equation}
\hor = \hor_{\sas}(Y) \oplus \hor_{\rie}(Y) \oplus \spn\{ Y \}.
\end{equation}
Hence, the horizontal trace of the Hessian can be decomposed, for any fixed $Y$, as the partial trace on each of the above components. In particular, outside of $\Cut_\ve(x)$, and letting $Y=\nabla_\hor r_\varepsilon$, we have
\begin{equation}\label{eq:sublapdec}
\Delta_\hor r_\varepsilon = \mathrm{Hess}(r_\varepsilon)(Y,Y) + \sum_{\alpha=1}^{m} \mathrm{Hess}(r_\varepsilon)(J_{Z_\alpha}Y,J_{Z_\alpha}Y) +  \sum_{i=1}^{n-m-1}\mathrm{Hess}(r_\varepsilon)(X_i,X_i).
\end{equation}
where $X_1,\dots,X_{n-m-1}$ is an orthonormal frame for  $\hor_{\rie}(Y)$, and $Z_1,\dots,Z_{m}$ is a suitably normalized orthogonal frame for $\Gamma(\ver)$.

Therefore, one can easily obtain an uniform sub-Laplacian comparison theorem for $\Delta_\hor r_\varepsilon$ by applying separately Theorems \ref{t:geoddir}, \ref{t:Riemannian-average}, \ref{t:SasakianZpar} and \ref{t:SasakianZperp} to each term in \eqref{eq:sublapdec}.

\begin{theorem}\label{t:horlapuniform}
Let $(\M,\hor,g)$ be an H-type foliation with parallel horizontal Clifford structure, satisfying the $J^2$ condition. Let $\kappa^2$ be the curvature of the vertical leaves. Assume that there exists $\rho \in \R$ such that
\begin{equation}
\mathrm{Sec}(X \wedge Y) \geq  \rho,\qquad \forall X,Y\in \hor.
\end{equation}
Fix $\ve >0$. Then, for all $y\notin \Cut_\ve(x)$ with $\nabla_\hor r_\ve(y) \neq 0$, it holds
 \begin{multline}\label{comparison:lsd}
\Delta_\hor r_\ve(y) \leq \frac{1-\|\nabla_\hor r_\varepsilon\|^2}{r_\varepsilon} + (n-m-1) F_{\mathrm{Riem}}(r_\varepsilon,K) \\
+ F_{\mathrm{Sas}}(r_\varepsilon,K_1) +  (m-1)F_{\mathrm{Sas}}(r_\varepsilon,K_2),
 \end{multline}
where
\begin{align}
 K &= \rho \|\nabla_\hor r_\varepsilon\|^2 +\tfrac{1}{4}\|\nabla_\ver r_\varepsilon\|^2, \\
K_1 & = \rho \|\nabla_\hor r_\varepsilon\|^2 + \|\nabla_\ver r_\varepsilon\|^2, \\
K_2 & = \rho \|\nabla_\hor r_\varepsilon\|^2 +(2-\kappa\varepsilon)(\kappa\varepsilon-1)\|\nabla_\ver r_\varepsilon\|^2.
\end{align}
\end{theorem}

As an immediate consequence we have the following sub-Riemannian comparison theorem. We stress that the assumptions and the conclusion of the statement depend only on the restriction of $g$ to $\hor$, that is on the induced sub-Riemannian structure. In particular the curvature of the leaves $\kappa$ plays no role.
\begin{theorem}\label{t:horlapuniform0}
Let $(\M,\hor,g)$ be an H-type foliation with parallel horizontal Clifford structure, satisfying the $J^2$ condition. Assume that there exists $\rho \in \R$ such that
\begin{equation}
\mathrm{Sec}(X \wedge Y) \geq  \rho,\qquad \forall X,Y\in \hor.
\end{equation}
Then for all $y\notin \Cut_0(x)$ it holds
 \begin{equation}
 \Delta_\hor r_0(y) \leq  (n-m-1) F_{\mathrm{Riem}}(r_0,K) + F_{\mathrm{Sas}}(r_0,K_1) +  (m-1)F_{\mathrm{Sas}}(r_0,K_2),
 \end{equation}
where
\begin{align*}
 K &= \rho +\tfrac{1}{4}\|\nabla_\ver r_0\|^2, \\
K_1 & = \rho  + \|\nabla_\ver r_0\|^2, \\
K_2 & = \rho -2\|\nabla_\ver r_0\|^2.
\end{align*}
\end{theorem}

\begin{proof}
By Proposition \ref{p:cutapprox}, on $\M \setminus \Cut_0(x)$ it holds in particular
\begin{equation}
\lim_{\ve \to 0} \nabla_\hor r_\ve  = \nabla_\hor r_0, \qquad
\lim_{\ve \to 0} \nabla_\ver r_\ve  = \nabla_\ver r_0, \qquad
\lim_{\ve \to 0} \Delta_\hor r_\ve  = \Delta_\hor r_0,
\end{equation}
Since on $\M \setminus \Cut_0(x)$ it also holds $\|\nabla_\hor r_0\| =1$, we can take the limit in Theorem \ref{t:horlapuniform}, which yields the statement.
\end{proof}

We note that the comparison function on the right hand side of \eqref{comparison:lsd} can be bounded from above by a function depending on $r_0$ only. Indeed, in the previous theorem, for simplicity assume that $\rho=0$. Then, 
\begin{equation}
F_{\mathrm{Riem}}(r_0,K) =F_{\mathrm{Riem}}(r_0,\tfrac{1}{4}\|\nabla_\ver r_0\|^2) \le \frac{1}{r_0},
\end{equation}
\begin{equation}
F_{\mathrm{Sas}}(r_0,K_1)=F_{\mathrm{Sas}}(r_0,\|\nabla_\ver r_0\|^2)\le \frac{4}{r_0},
\end{equation}
and from Remark \ref{vertical gradient comparison}
\begin{equation}
F_{\mathrm{Sas}}(r_0,K_2)=F_{\mathrm{Sas}}(r_0, -2\|\nabla_\ver r_0\|^2)\le F_{\mathrm{Sas}}\left(r_0, -\frac{4\pi^2}{r_0^2}\right)\le \frac{C}{r_0},
\end{equation}
where $C>0$ is an explicit constant whose value is given by
\begin{equation}
C=\pi  \left(\coth (\pi )+\frac{\pi }{\pi  \coth (\pi )-1}\right) >4.
\end{equation}
As a consequence, one obtains the following corollary (which is only sharp when $m=1$).

\begin{corollary}
Let $(\M,\hor,g)$ be an H-type foliation with parallel horizontal Clifford structure, satisfying the $J^2$ condition, and with non-negative horizontal Bott curvature. Then, there exists a constant $C>4$ such that, outside of $\Cut_0(x)$, it holds
\begin{equation}
\Delta_\hor r_0 \le\frac{n-m+3+C(m-1)}{r_0}.
\end{equation}
\end{corollary}


\subsection{Removing the \texorpdfstring{$J^2$}{J2} condition}

Here we study a class of structures where the $J^2$ condition does not necessarily hold. In particular, for H-type foliations with completely parallel torsion and non-negative horizontal sectional curvature, we  achieve the sharp sub-Laplacian comparison result:
\begin{equation}\label{eq:htypeLCT}
\Delta_\mathcal{H} r_0 \le \frac{n+3m-1}{r_0},
\end{equation}
outside of the cut locus. The proofs are mostly minor modifications of the ones seen in the previous sections.

Such comparison theorem  applies in the case of H-type groups. We remark that the estimate \eqref{eq:htypeLCT} is equivalent to the measure contraction property $\mathrm{MCP}(0,N)$ for all $N\geq n+3m-1$ of H-type Carnot groups, proven in \cite{Rizzi18}.

\subsubsection{A finer splitting}

Fix a non-zero $Y\in T\M$ (usually $Y=\dot\gamma$ is the tangent vector to a geodesic, in which case the definition that follow make sense along $\gamma$). If the $J^2$ condition does not hold, the splitting introduced in Section \ref{s:splitting} is no longer $\hat{\nabla}^\varepsilon$-parallel. 
We introduce thus a finer splitting. Let first
\begin{equation}\label{eq:splittingver-finer}
\ver = \ver_{\hty}(Y) \oplus \ver_{\sas}(Y),
\end{equation}
where each subspace is defined as
\begin{align}
\ver_{\sas}(Y) &= \{ Z \in \ver \mid J_{Y}J_Z Y \subset J_{\V}(Y) \oplus \spn\{Y_\hor\} 	\}, \\
\ver_{\hty}(Y) & = \{ Z \in \ver \mid J_{Y}J_Z Y \perp J_{\V}(Y) \oplus \spn\{Y_\hor\}, \text{ and } J_Y J_Z Y \neq 0\}.
\end{align}
Likewise, one has a splitting of the horizontal space
\begin{equation}\label{eq:splittinghor-finer}
\hor = \hor_{\hty}(Y) \oplus \hor_{\sas}(Y) \oplus \hor_{\rie}(Y) \oplus \spn\{ Y_\hor \},
\end{equation}
where each subspace is defined as
\begin{align}
\hor_{\sas}(Y) &= \{J_Z Y \mid Z \in \ver_{\sas}(Y)\}, \\
\hor_{\hty}(Y) & = \{J_Z Y, \, J_{Y}J_Z Y \mid Z \in \ver_\hty(Y) \},\\
\hor_{\rie}(Y) &= \{X \in \hor \mid X \perp \hor_{\sas}(Y)\oplus \hor_{\hty}(Y)\oplus  \spn\{Y_\hor\} \}.
\end{align}
\begin{remark}
If the $J^2$ condition holds then $\hor_{\hty}(Y) = \ver_{\hty} (Y)= \emptyset$, and $\hor_{\sas}(Y) = J_{\V}(Y)$. Thus \eqref{eq:splittingver-finer} is trivial and \eqref{eq:splittinghor-finer} reduces  to the one introduced in Section \ref{s:splitting}.
\end{remark}
\begin{remark}
When the $J^2$ condition does not hold, the dimensions of each subspace may depend on $Y$. We remark that $J_Y$ preserves $\hor_\hty(Y)$, $\hor_\sas(Y)\oplus \spn\{ Y_\hor \}$, and thus $\hor_\rie(Y)$. It follows that for some $p,q,\ell \geq 0$ it holds
\begin{equation}
\dim \hor_\hty(Y) = 2p, \qquad \dim\hor_\sas(Y) = 2q+1, \qquad \dim\hor_\rie(Y) = 2\ell,
\end{equation}
and thus $n = 2(p+q+\ell +1)$. Here $p,q,\ell$ might depend on $Y$.
\end{remark}
The next statement is a generalization of Proposition \ref{p:parallelism}.
\begin{proposition}\label{p:parallelism-finer}
Let $(\M,\hor,g)$ be an H-type foliation with parallel horizontal Clifford structure. Let $\gamma:[0,r_\varepsilon] \to \M$ be a non-trivial $g_\varepsilon$-geodesic. The splitting \eqref{eq:splittingver-finer}-\eqref{eq:splittinghor-finer} associated with $\dot{\gamma}$ is orthogonal and $\hat{\nabla}^\varepsilon_{\dot\gamma}$-parallel. Hence, each subspace has constant dimension along $\gamma$.
\end{proposition}
\begin{proof}
The case in which $\dot\gamma \in \hor$ is trivial. So we suppose $\dot\gamma_\ver \neq 0$ and, without loss of generality, that $\|\dot\gamma_\hor\|=1$. For simplicity in this proof we set $\ver_\sas = \ver_\sas(\dot\gamma)$ and similarly for the other subspaces.

We first show that the splitting is orthogonal. The map $Z\mapsto J_{\dot\gamma}J_Z\dot\gamma$ is injective, so $\dim \ver_\hty + \dim \ver_\sas = \dim \ver$. Furthermore, if $Z \in \ver_\sas$ and $W \in \ver_\hty$ we have
\begin{equation}
\langle W, Z\rangle = \langle J_{\dot\gamma} J_Z \dot\gamma,J_{\dot\gamma}J_W \dot\gamma \rangle =0,
\end{equation}
hence $\ver_{\hty}\perp \ver_{\sas}$. The orthogonality of the horizontal splitting follows by  definition.

We prove that the splitting is $\hat{\nabla}^\varepsilon_{\dot\gamma}$-parallel (simply ``parallel'', in the following). Using the parallel horizontal Clifford structure and the H-type assumptions, we have
\begin{align*}
(\hat{\nabla}^\varepsilon_{\dot\gamma}J)_Z = (\nabla_{\dot\gamma}J)_Z + [J^\varepsilon_{\dot\gamma},J_Z] =  \left(\frac{2}{\varepsilon}-\kappa\right)\left(J_{\dot\gamma}J_Z + \langle\dot\gamma,Z\rangle\mathbbold{1}_\hor\right).
\end{align*}
In particular $(\hat{\nabla}^\varepsilon_{\dot\gamma}J)_{\dot\gamma} = 0$. Let $W$ be $\hat{\nabla}^\varepsilon_{\dot\gamma}$-parallel. Let also $\xi$ denote a generic vertical vector field along the geodesic, which we may assume to be parallel as well. Using the above relation, we have
\begin{equation}
\frac{d}{dt} \langle J_{\dot\gamma} J_W \dot\gamma, J_\xi\dot\gamma\rangle = 0, \qquad \frac{d}{dt} \langle J_{\dot\gamma} J_W \dot\gamma, \dot\gamma\rangle = 0.
\end{equation}
In particular, if $W \in \ver_{\hty}$ at the initial time, it will remain in $\ver_{\hty}$  for all times. This proves that $\ver_{\hty}$ is parallel. Since $\hat{\nabla}^\varepsilon_{\dot\gamma}$ is metric and preserves the vertical space, it follows that $\ver_{\sas}$ is parallel as well.

Consider one of the generators of $\hor_{\sas}$, that is $J_Z\dot\gamma$, with $Z\in \ver_\sas$. We can assume the latter to be $\hat{\nabla}^\varepsilon$-parallel, by the previous step of the proof. If $Z$ is proportional to $\dot\gamma_\ver$, then $\hat{\nabla}^\varepsilon_{\dot\gamma} J_Z\dot\gamma=0$. On the other hand, if $Z$ is orthogonal to $\dot\gamma_\ver$, we have
\begin{equation}
\hat{\nabla}^\varepsilon_{\dot\gamma} J_Z\dot\gamma =  \left(\frac{2}{\varepsilon}-\kappa\right) J_{\dot\gamma}J_Z \dot\gamma =  \left(\frac{2}{\varepsilon}-\kappa\right)\left(J_{\xi_{\sas}}\dot\gamma + J_{\xi_{\hty}}\dot\gamma + \alpha \dot\gamma_\hor\right),
\end{equation}
for some $\xi_{\hty} \in \ver_\hty$, $\xi_{\sas} \in \ver_\sas$, and $\alpha \in \R$. The last equality follows from the definition of $\hor_{\sas}$. We easily deduce that $\alpha =0$ and $\xi_{\hty} = 0$. Therefore $\hor_{\sas}$ is parallel.

Similarly, if $J_W\dot\gamma$ and $J_{\dot\gamma}J_W\dot\gamma$ are generators of $\hor_\hty$, for $W\in\ver_\hty$, and recalling that $W\perp \dot\gamma$ in this case, we have
\begin{align*}
\hat{\nabla}^\varepsilon_{\dot\gamma} J_W\dot\gamma  & =  \left(\frac{2}{\varepsilon}-\kappa\right) J_{\dot\gamma} J_W \dot\gamma, \\
\hat{\nabla}^\varepsilon_{\dot\gamma} J_{\dot\gamma} J_W\dot\gamma  & = \left(\frac{2}{\varepsilon}-\kappa\right) J_{\dot\gamma}^2 J_W \dot\gamma = - \left(\frac{2}{\varepsilon}-\kappa\right)J_W\dot\gamma.
\end{align*}
It follows that $\hor_{\hty}$ is parallel. Since $\hat{\nabla}^\varepsilon$ is metric, and preserves $\hor$, then also $\hor_{\rie}$ must be parallel.
\end{proof}

\subsubsection{Hessian comparison theorems}

We state here versions of the comparison theorems without the $J^2$ condition. For the case of Riemannian directions there is mostly no difference with respect to Theorem \ref{t:Riemannian-sect}. The assumption is somewhat more involved, due to the fact that now the spaces in the splitting depend also on the vertical component of the vector field $Y$ determining it. If $x\in \M$ is fixed and $y\notin\Cut_\ve(x)$, so that $r_\ve = d_\ve(x,\cdot)$ is smooth in a neighbourhood of $y$, we adopt the shorthand
\begin{equation}
\hor_{\rie}(y):= \hor_{\rie}(\nabla r_\ve(y)), \qquad \ver_{\rie}(y):= \ver_{\rie}(\nabla r_\ve(y)),
\end{equation}
and similarly for all other subspaces.

\begin{theorem}
Let $(\M,\hor,g)$ be an H-type foliation with parallel horizontal Clifford structure. Assume that there exists $\rho \in \R$ such that
\begin{equation}
\mathrm{Sec}(X \wedge Y_\hor) \geq \rho,\qquad \forall\, Y\in T\M, \quad X \in \hor_\rie(Y).
\end{equation}
Let $y\notin \Cut_\ve(x)$ with $\nabla_\hor r_\ve(y) \neq 0$. Let $X\in \hor_{\rie}(y)$, with $\|X\|=1$. Then at $y$ it holds
\begin{equation}
\mathrm{Hess}(r_\varepsilon)(X,X) \leq  F_{\mathrm{Riem}}(r_\varepsilon,K), \qquad K = \rho \|\nabla_\hor r_\varepsilon\|^2 + \tfrac{1}{4}\|\nabla_\ver r_\varepsilon\|^2.
\end{equation}
Provided that $K>0$, it follows that
\begin{equation}
r_\varepsilon(y) < \frac{\pi}{\sqrt{\rho \|\nabla_\hor r_\varepsilon\|^2 + \tfrac{1}{4}\|\nabla_\ver r_\varepsilon\|^2}}.
\end{equation}
\end{theorem}
\begin{proof}
The proof is analogous to the one of Theorem \ref{t:Riemannian-sect}. The only ingredient was Proposition \ref{p:parallelism}, which is now replaced by the more general Proposition \ref{p:parallelism-finer} when the $J^2$ condition does not hold.
\end{proof}

For the special horizontal vectors $J_{Z}\nabla_\hor r_\varepsilon$ with $Z \propto \nabla_\ver r_\ve$, which indeed belongs to $\hor_{\sas}(\nabla r_\varepsilon)$, Theorem \ref{t:SasakianZpar} holds unchanged.

Consider then the remaining Sasakian directions $X= J_Z \nabla r_\varepsilon$ for $Z\in \ver_\sas$, and $Z\perp \nabla_\ver r_\varepsilon$. We remind that these directions appear in pairs $J_Z \nabla r_\varepsilon$, $J_{\nabla_\ver r_\varepsilon}J_Z \nabla r_\varepsilon$, which are both of the form $J_{Z'} \nabla r_\varepsilon$ for some $Z'$ depending on $\nabla r_\varepsilon$. In this part of the space, it is as the $J^2$ condition was verified, and we have the following result.

\begin{theorem}
Let $(\M,\hor,g)$ be an H-type foliation with parallel horizontal Clifford structure. Let $\kappa^2$ be the curvature of the vertical leaves. Assume that there exists $\rho \in \R$ such that
\begin{equation}
\mathrm{Sec}(X \wedge J_ZX) \geq \rho,\qquad \forall X\in \hor,\, Z\in \ver.
\end{equation}
Let $y \notin \Cut_\ve(x)$, with $\nabla_\hor r_\ve(y) \neq 0$. Let $Z\in \ver_{\sas}(y)$ be non-zero with $Z\perp \nabla_\ver r_\varepsilon(y)$. Set
\begin{equation}
X = \frac{J_Z \nabla_\hor r_\varepsilon}{\|Z\|\| \nabla_\hor r_\varepsilon\|}.
\end{equation}
Then, at $y$, it holds
\begin{equation}
\mathrm{Hess}(r_\varepsilon)(X,X) \leq  F_{\mathrm{Sas}}(r_\varepsilon,K_2),
\end{equation}
where
\begin{equation}
K_2 = \rho \|\nabla_\hor r_\varepsilon\|^2 +(2-\kappa\varepsilon)(\kappa\varepsilon-1)\|\nabla_\ver r_\varepsilon\|^2.
\end{equation}
If, furthermore, $\rho = \kappa = 0$, one has the sharper bound
\begin{equation}
\mathrm{Hess}(r_\varepsilon)(X,X) \leq  \frac{v}{2} \cot\left(\frac{v r_\varepsilon}{2}\right) + v \frac{1-\cos(v r_\varepsilon)}{v r_\varepsilon- \sin(vr_\varepsilon)} \leq \frac{4}{r_\varepsilon}.
\end{equation}
\end{theorem}
\begin{proof}
Let $\gamma :[0,r_\varepsilon] \to \M$ be the $g_\varepsilon$-geodesic between $x$ and $y$. Assume first that $\|\nabla_\ver r_\ve\| >0$. In this case the condition $Z \in \ver_\sas(\dot\gamma)$ with $Z \perp \dot\gamma$ implies that $V= T(\dot\gamma,J_{\dot\gamma} J_Z \dot\gamma)$ is a non-zero vector in $\ver_\sas(\dot\gamma)$ and $V \perp \dot\gamma$. In particular one can choose a $\hat{\nabla}^\varepsilon$-parallel $Z \in \ver_\sas(\dot\gamma)$ such that $J_Z Y$, $J_Y J_Z Y$, $Z$ and $V$ are independent along $\gamma$. The proof of the first part of the theorem is thus identical to the one of Theorem \ref{t:SasakianZperp}. If instead $\nabla_\ver r_\ve(y) = \emptyset$ then we again proceed as in the analogous particular case in the proof of Theorem \ref{t:SasakianZperp}, for which the introduction of $V$ is not necessary.

If $\rho=\kappa=0$, then we can explicitly solve the exact system \eqref{eq:completesystem1}-\eqref{eq:completesystem2} in the proof of Theorem \ref{t:SasakianZperp} instead of looking for a simplified system as in the previous case. This yields the result upon application of Theorem \ref{t:comparison}. We omit the details.
\end{proof}

For H-type directions we have a new statement. It it stated only for completely parallel torsion and non-negative curvature in the relevant sections. We notice that if $Y_\ver =0$, then $\hor_{\hty}(Y) = \ver_{\hty}(Y) = 0$. Therefore in the next theorem we must assume $\nabla_\ver r_\ve(y) \neq 0$ to exclude a vacuous statement.

\begin{theorem}
Let $(\M,\hor,g)$ be an H-type foliation with completely parallel torsion. Assume that 
\begin{equation}
\mathrm{Sec}(X \wedge Y_\hor) \geq 0\qquad  \forall\, Y\in T\M, \quad X \in \hor_{\hty}(Y).
\end{equation}
Let $y\notin \Cut_\ve(x)$, with $\nabla_\hor r_\ve(y) \neq 0$ and $\nabla_\ver r_\ve(y) \neq 0$. For a non-zero $Z\in \ver_{\hty}(y)$, let
\begin{equation}
X = \frac{J_Z \nabla_\hor r_\varepsilon}{\|Z\|\| \nabla_\hor r_\varepsilon\|}, \qquad Y = \frac{J_{\nabla_\ver r_\varepsilon} J_Z \nabla_\hor r_\varepsilon}{\|\nabla_\ver r_\varepsilon\| \|Z\|\| \nabla_\hor r_\varepsilon\|},
\end{equation}
Then at $y$ it holds
\begin{align*}
\mathrm{Hess}(r_\varepsilon)(X,X) &  \leq  \frac{v}{2}\cot\left(\frac{v r_\varepsilon}{2}\right) - 2 v \frac{\sin\left(\frac{v r_\varepsilon}{2}\right)}{\sin(v r_\varepsilon) - v r_\varepsilon}, \\
\mathrm{Hess}(r_\varepsilon)(Y,Y) &  \leq   \frac{v}{2}\cot\left(\frac{v r_\varepsilon}{2}\right) - \frac{ v}{2} \frac{\sin\left(\frac{v r_\varepsilon}{2}\right)}{\sin(v r_\varepsilon) - v r_\varepsilon},
\end{align*}
where $v= \|\nabla_\ver r_\varepsilon(y)\|$. In particular, one always has
\begin{equation}
\mathrm{Hess}(r_\varepsilon)(X,X) + \mathrm{Hess}(r_\varepsilon)(Y,Y) \leq \frac{5}{r_\varepsilon}.
\end{equation}
\end{theorem}
\begin{proof}
Let $\gamma :[0,r_\varepsilon] \to \M$ be the $g_\varepsilon$-geodesic between $x$ and $y$. Let $Z\in \hor_\hty(\dot\gamma)$ be $\hat{\nabla}^\varepsilon$-parallel, and $\|Z\|=1$. The condition $Z \in \ver_\hty(\dot\gamma)$ implies that $Z \perp \dot\gamma$ and $T(\dot\gamma,J_{\dot\gamma} J_Z \dot\gamma) =0$. Taking this into account, the proof is then similar to the one of Theorem \ref{t:SasakianZperp}: we consider the $g$-orthonormal vector fields
\begin{equation}
X = \frac{J_Z\dot\gamma}{h}, \qquad Y = \frac{J_{\dot\gamma}J_Z\dot\gamma}{h\varepsilon v}, \qquad Z.
\end{equation}
We look for a solution of the Jacobi equation in the constant curvature case, of the form
\begin{equation}
W(t)= a(t) X  + c(t) Y + b(t) Z, \qquad \forall t\in [0,r_\varepsilon],
\end{equation}
such that $W(0)=0$ and $W(r_\varepsilon) = X$ (for the first part of the statement) or $W(r_\varepsilon)= Y$ (for the second part of the statement). The computations are identical to the one in the proof of Theorem \ref{t:SasakianZperp}, formally setting $d=0$ (and, of course, $\kappa=\rho=0$). We obtain
\begin{align*}
\ddot{a} -3 v\dot{c} + \tfrac{h}{\varepsilon}(\dot{b}-ha) -v^2 a =0,\\
\ddot{c} +3\dot{a} + \tfrac{h}{\varepsilon}(\dot{d}-h c) -v^2 c =0,\\
\ddot{b} - h \dot{a} = 0.
\end{align*}
The system is easily solved, for both cases of boundary conditions, and provided that $r_\varepsilon$ is sufficiently small. Applying Theorem \ref{t:comparison}, we obtain
\begin{align*}
\mathrm{Hess}(r_\varepsilon)\left(\frac{J_Z\dot\gamma}{h},\frac{J_Z\dot\gamma}{h}\right) & \leq \frac{v}{2}\cot\left(\frac{v r_\varepsilon}{2}\right) - 2 v \frac{\sin\left(\frac{v r_\varepsilon}{2}\right)}{\sin(v r_\varepsilon) - \left(1+\tfrac{2\varepsilon v^2}{h^2}\right)v r_\varepsilon}\\
& \leq \frac{v}{2}\cot\left(\frac{v r_\varepsilon}{2}\right) - 2 v \frac{\sin\left(\frac{v r_\varepsilon}{2}\right)}{\sin(v r_\varepsilon) - v r_\varepsilon}
\end{align*}
in the first case. In the second case, we get
\begin{align*}
\mathrm{Hess}(r_\varepsilon)\left(\frac{J_{\dot\gamma}J_Z\dot\gamma}{h \varepsilon v},\frac{J_{\dot\gamma}J_Z\dot\gamma}{h \varepsilon v}\right) & \leq \frac{v}{2}\cot\left(\frac{v r_\varepsilon}{2}\right) -  \frac{v}{2} \frac{\sin\left(\frac{v r_\varepsilon}{2}\right)}{\sin(v r_\varepsilon) - \left(1+\tfrac{2\varepsilon v^2}{h^2}\right)v r_\varepsilon}\\
& \leq \frac{v}{2}\cot\left(\frac{v r_\varepsilon}{2}\right) - \frac{ v}{2} \frac{\sin\left(\frac{v r_\varepsilon}{2}\right)}{\sin(v r_\varepsilon) - v r_\varepsilon}.
\end{align*}
Both results imply that $r_\varepsilon v < 2\pi$, for all $\varepsilon>0$.
\end{proof}

\subsubsection{Sub-Laplacian comparison theorems}

Combining the results of the previous subsections and arguing as before one eventually obtains the following sub-Laplacian comparison theorem which is sharp on H-type groups.

\begin{theorem}
Let $(\M,\hor,g)$ be an H-type foliation with completely parallel torsion. Assume that 
\begin{equation}
\mathrm{Sec}(X \wedge Y) \geq 0\qquad  \forall\, X,Y\in \hor.
\end{equation}
For $y\notin \Cut_\ve(x)$, with $\nabla_\hor r_\ve(y) \neq 0$  it holds
\begin{equation}
\Delta_\hor r_\ve(y) \leq \frac{n+3m-\|\nabla_\hor r_\varepsilon\|^2}{r_\varepsilon}.
\end{equation}
Thus, for all $y\notin \Cut_0(x)$ it holds
\begin{equation}
\Delta_\hor r_0(y) \leq \frac{n+3m-1}{r_0}.
\end{equation}
\end{theorem}

\section{A sharp sub-Riemannian Bonnet-Myers theorem for general H-type sub-Riemannian manifolds}\label{s:sharpBM}

In this section we prove a sharp Bonnet-Myers type theorem for structures which are not necessarily foliations. This result is quite general, but it holds only in the sub-Riemannian limit  (and not, as in the previous section, uniformly for all $\ve \geq 0$).

Consider a general Riemannian manifold $(\M,g)$ together with a vector bundle orthogonal splitting $T\M = \hor \oplus \ver$. The vertical bundle $\ver$ is not required to be integrable nor metric. The Hladky connection is well defined, and Theorem \ref{t:index} holds, with $D = \hat\nabla^\ve$, given explicitly by
\begin{equation}
\hat{\nabla}^\varepsilon_X Y = \nabla_X Y + J^\varepsilon_X Y, \qquad \forall X,Y \in \Gamma(T\M).
\end{equation}
where $J^\varepsilon$ is the operator defined as in Section \ref{s:definitions}, with respect to the metric $g_\varepsilon$. For a $g_\ve$-geodesic $\gamma : [0,r_\ve] \to \M$, consider the $n-m-1$-vector subspace of $\hor$ along $\gamma$ given by
\begin{equation}
\mathfrak{L}_J(\dot\gamma) = \spn\{X \in \hor \mid X \perp J_\ver \dot\gamma,\; X \perp \dot\gamma\}.
\end{equation}
The idea is that the index form behaves on this space as in the Riemannian case, up to corrections of order $\ve$ (which are negligible in the limit). Notice that $\mathfrak{L}_J(\dot\gamma)$ is defined as $\hor_\rie(\dot \gamma)$ of Section \ref{s:splitting}, but since the setting is more general we prefer to adopt a different notation, to avoid confusion.

\medskip

We assume that the following conditions are satisfied.

\begin{enumerate}[(i)]
\item the H-type condition: $J_Z^2 = - \|Z\|^2 \mathbbold{1}_{\hor}$ for all $Z\in \Gamma(\ver)$;
\item the $J^2$ condition, in the sense of \cite{Cowling-Htype,CalinChangMarkina}: for all $Z,Z'\in \Gamma(\ver)$, $X \in \Gamma(\hor)$, with $\langle Z,Z'\rangle =0$ there exists $Z''\in \Gamma(\ver)$ such that 
\begin{equation}
J_Z J_{Z'}X = J_{Z''}X.
\end{equation}
\item for all $Y \in \Gamma(\hor)$ and $Z\in\Gamma(\ver)$, we have
\begin{equation}
\sum_{i=1}^{n-m-1} \langle(\nabla_{X_i} J)_Z X_i,Y\rangle =0,
\end{equation}
where $X_1,\dots,X_{n-m-1}$ is an orthonormal frame for $\mathfrak{L}_J(Y)$.
\end{enumerate}

\begin{theorem}\label{t:BMII}
Let $(\M,\hor,g)$ be a complete sub-Riemannian structure with corank $m$ and rank $n > m+1$, satisfying assumptions (i)-(ii)-(iii). Assume that there exists $\rho >0$ such that, for all unit $X \in \hor$ it holds
\begin{multline}\label{eq:curvassII}
\Ric^\nabla_\hor(X,X) - \sum_{\alpha =1}^m R^\nabla(X,J_{Z_\alpha}X,J_{Z_\alpha} X,X )  \\
- \sum_{\alpha=1}^m\left(\|(\nabla_X J)_{Z_\alpha}X\|^2 - \sum_{\beta=1}^m \langle (\nabla_X J)_{Z_\alpha}X,J_{Z_\beta} X\rangle^2 \right) \geq (n-m-1) \rho,
\end{multline}
where $Z_1,\dots,Z_m \in \ver$ is any orthonormal frame. Then $\M$ is compact with sub-Rie\-mannian diameter not greater than $\pi/\sqrt{\rho}$, and the fundamental group is finite.
\end{theorem}
\begin{remark}
Assumptions (i)-(ii)-(iii) are verified for any contact and quaternionic contact structures sub-Riemannian structures, as in \cite{BRcontact} and \cite{BI17}, respectively. We refer to these reference for more precise definitions. We only mention that the verification of (i) and (ii) are trivial. Furthermore, in the contact case, (iii) follows from \cite[Lemma 6.8, (i) and (e)]{BRcontact}). In the quaternionic contact case, condition (iii) can be checked using equivalently the Biquard connection $\bar{\nabla}$, since it holds $\bar{\nabla}_X Y = \nabla_X Y$ for all $X,Y\in\Gamma(\hor)$. Therefore, (iii) holds by definition of Biquard connection \cite[Thm.\ 4.2.4 (v)]{SV-QC}. Theorem \ref{t:BMII} thus includes the contact Bonnet-Myers theorem \cite[Thm.\ 1.7]{BRcontact}, and the quaternionic contact one of \cite{BI17}.
\end{remark}
Before proving Theorem \ref{t:BMII} we need two lemmas. We recall that the operator $J_Z$ is defined as in Section \ref{s:definitions}, with respect to the Riemannian metric $g$ and the splitting $T\M = \hor \oplus \ver$. We denote by $J^\varepsilon$ the analogous operator associated with the metric $g_\varepsilon$.

\begin{lemma}\label{l:homo}
Under the H-type condition, the following homogeneity properties hold:
\begin{equation}
J^\varepsilon_\hor =\varepsilon J_\hor \qquad J^\varepsilon_\ver = \frac{1}{\varepsilon} J_\ver.
\end{equation}
Furthermore we have
\begin{equation}
J^\varepsilon_\hor \hor \subseteq \ver \qquad J^\varepsilon_\ver \hor \subseteq \hor, \qquad J^\varepsilon_\ver \ver = 0.
\end{equation}
In particular, as $\varepsilon \to 0$, the only element that does not vanish uniformly is  $J^\varepsilon_\ver \hor$, exactly as in the case of a totally geodesic foliation
\end{lemma}
\begin{proof}
The H-type condition implies $J_\ver \ver = 0$, which is equivalent to $(\mathcal{L}_{\hor} g)(\ver,\ver) =0$. This condition is the analogous of the bundle-like metric condition for foliations, even though we are not assuming that the vertical bundle is a foliation. The proof now follows by unravelling the definitions of $J^\varepsilon$ and $g_\varepsilon$.
\end{proof}

\begin{lemma}\label{l:frame}
Let $\gamma :[0,r_\varepsilon] \to \M$ be a $g_\varepsilon$-geodesic. Assume $n-m-1>0$. There exists a moving frame $X_1^,\ldots,X_{n-m-1}$ of $\hor$ along $\gamma$ with the following properties:
\begin{enumerate}[(a)]
\item It is a $g_\varepsilon$-orthonormal frame;
\item It is a frame for $\mathfrak{L}_J(\dot\gamma)$;
\item It is almost $\hat\nabla^\varepsilon$-parallel, that is the projection of $\hat\nabla^\varepsilon_{\dot\gamma} X_i$ on $\mathfrak{L}_J(\dot\gamma)$ is zero.
\end{enumerate}
The frame is unique up to a constant orthogonal transformation, and it is a solution of
\begin{equation}\label{eq:odeparallel}
\hat\nabla^\varepsilon_{\dot\gamma} X_i = -  \frac{1}{\|\dot\gamma_\hor\|^2} \sum_{\alpha=1}^m \langle X_i ,\hat\nabla^\varepsilon_{\dot\gamma} (J_{Z_\alpha}\dot\gamma) \rangle J_{Z_\alpha}\dot\gamma - 	\sum_{\alpha=1}^m \langle X_i, \hat\nabla^\varepsilon_{\dot\gamma} Z_\alpha\rangle Z_\alpha,
\end{equation}
where $Z_1,\dots,Z_m $ is any $g$-orthonormal moving frame of $\V$ along $\gamma$.
\end{lemma}
\begin{proof}
For a frame obeying (b), one must have for all $\alpha =1,\dots,m$
\begin{equation}
\frac{d}{dt} \langle X_i,\dot\gamma\rangle = \frac{d}{dt} \langle X_i,Z_\alpha \rangle=  \frac{d}{dt} \langle X_i,J_{Z_\alpha} \dot\gamma\rangle = 0.
\end{equation}
Since  $\hat\nabla^\varepsilon_{\dot\gamma}\dot\gamma =0$, these equation are
\begin{gather} \nonumber
\langle \hat{\nabla}^\varepsilon_{\dot\gamma} X_i,\dot\gamma\rangle = 0, \qquad \langle \hat{\nabla}^\varepsilon_{\dot\gamma} X_i,Z_\alpha\rangle= -\langle X_i,\hat{\nabla}^\varepsilon_{\dot\gamma} Z_\alpha\rangle , \\ \nonumber \langle \hat{\nabla}^\varepsilon_{\dot\gamma} X_i, J_{Z_\alpha} \dot\gamma \rangle = -\langle X_i,  \hat{\nabla}^\varepsilon_{\dot\gamma} (J_{Z_\alpha}\dot\gamma) \rangle.
\end{gather}
Taking into account also condition (c), and recalling that $\{J_{Z_\alpha}\dot\gamma/\|\dot\gamma_\hor\|\}_{\alpha=1}^m$ constitutes a $g$-orthonormal set perpendicular to $\mathfrak{L}_J(\dot\gamma)$ and $\dot\gamma_\hor$, we obtain that the frame, if it exists, must be a solution of \eqref{eq:odeparallel}. The latter is linear and it has a unique solution defined on $[0,r]$ for any given initial condition $X_1(0),\dots, X_{n-m-1}(0)$, which we chose to be in $\mathfrak{L}_J(\dot\gamma)$. By construction, these solutions satisfy
\begin{equation}
\langle \hat{\nabla}^\varepsilon_{\dot\gamma}X_i,X_j\rangle = - \langle \hat{\nabla}^\varepsilon_{\dot\gamma}X_j,X_i\rangle, \qquad \forall i,j=1,\dots,n-m-1.
\end{equation}
In particular, (a) holds for all $t$ provided that the initial condition is orthonormal, and by construction also (b) and (c) hold for all $t$. This prove the existence of the frame. Finally, since \eqref{eq:odeparallel} is linear, its solutions are uniquely determined up to a constant linear transformation which, due to condition (a), must be orthogonal.
\end{proof}

\subsection{Proof of Theorem \ref{t:BMII}}

Fix $\varepsilon>0$. Let $x \in \M$ and $y \notin \Cut_\ve(x)$. Let $\gamma : [0,r_\varepsilon] \to \M$ be the unique $g_\varepsilon$-geodesic, parametrized with unit speed, joining $x$ with $y$. Let $X_1,\dots,X_{n-m-1}$ be the moving frame of Lemma \ref{l:frame} (it depends on $\varepsilon$). For $i=1,\dots,n-m-1$, let
\begin{equation}
W_i(t) = f(t) X_i(t), \qquad t \in [0,r_\varepsilon], 
\end{equation}
for some smooth non-negative function $f:[0,r_\varepsilon] \to \R$, with $f(0) = 0$ and $f(r_\varepsilon) = 1$, to be chosen later. By Theorem \ref{t:index}, with $D=\hat\nabla^\ve$, we have at $r_\ve$
\begin{equation}\label{eq:pretrace}
\mathrm{Hess}^{\hat\nabla^\varepsilon} r_\varepsilon (X_i,X_i) \leq \langle W_i(r_\varepsilon),\hat\nabla^\varepsilon_{\dot\gamma}W_i(r_\varepsilon)\rangle_\varepsilon \\ - \int_0^{r_\varepsilon} \langle \hat\nabla^\varepsilon_{\dot\gamma} \nabla^\varepsilon_{\dot\gamma}W_i+ \hat{R}^{\varepsilon}(W_i,\dot\gamma)\dot\gamma,W_i\rangle_\varepsilon dt.
\end{equation}
Using the properties of the frame, one has $\langle W_i(r_\varepsilon),\hat\nabla^\varepsilon_{\dot\gamma}W_i(r_\varepsilon)\rangle = \dot{f}(r_\varepsilon)$, for all $i$, and thus
\begin{equation}\label{eq:firstpiece}
\sum_{i=1}^{n-m-1} \langle W_i(r_\varepsilon),\hat\nabla^\varepsilon_{\dot\gamma}W_i(r_\varepsilon)\rangle_\varepsilon = (n-m-1)\dot{f}(r_\varepsilon).
\end{equation}
We focus now on the integrand in the r.h.s.\ of \eqref{eq:pretrace}:
\begin{equation}\label{eq:integrand}
i(W_i,W_i) =\langle \hat\nabla^\varepsilon_{\dot\gamma} \nabla^\varepsilon_{\dot\gamma}W_i,W_i\rangle_\varepsilon + \hat{R}^{\varepsilon}(W_i,\dot\gamma,\dot\gamma,W_i\rangle.
\end{equation}
Recall that $\hat{T}^\varepsilon$ is the torsion of $\hat\nabla^\varepsilon$ (which is completely skew-symmetric by Lemma \ref{l:constructionskew}) while $T$ is the torsion of $\nabla$. Dropping the index $i$, for the first term in \eqref{eq:integrand} we obtain
\begin{align} \nonumber
\langle \hat\nabla^\varepsilon_{\dot\gamma} \nabla^\varepsilon_{\dot\gamma} W_i,W_i\rangle_\varepsilon & = \langle \hat\nabla^\varepsilon_{\dot\gamma} \hat\nabla^\varepsilon_{\dot\gamma} W_i-\hat\nabla^\varepsilon_{\dot\gamma} (\hat{T}^\varepsilon(\dot\gamma,W_i)),W_i\rangle_\varepsilon \\ \nonumber
& =  \langle \hat\nabla^\varepsilon_{\dot\gamma} \hat\nabla^\varepsilon_{\dot\gamma} W_i,W_i\rangle_\varepsilon +\langle\hat{T}^\varepsilon(\dot\gamma,W_i),\hat\nabla^\varepsilon_{\dot\gamma} W_i\rangle_\varepsilon \\ \nonumber
& = f\ddot{f} - f^2\langle \hat\nabla^\varepsilon_{\dot\gamma} X_i,\hat\nabla^\varepsilon_{\dot\gamma} X_i\rangle_\varepsilon +f^2\langle \hat{T}^\varepsilon(\dot\gamma,X_i),\hat\nabla^\varepsilon_{\dot\gamma} X_i\rangle_\varepsilon \\ \nonumber
& = f\ddot{f} - f^2\langle \hat\nabla^\varepsilon_{\dot\gamma} X_i - T(\dot\gamma,X_i) - J^\varepsilon_{\dot\gamma}X_i+ J^\varepsilon_{X_i}\dot\gamma, \hat\nabla^\varepsilon_{\dot\gamma} X_i\rangle_\varepsilon\\
& = f\ddot{f} - f^2\langle \hat\nabla^\varepsilon_{\dot\gamma} X_i - J^\varepsilon_{\dot\gamma}X_i+ J^\varepsilon_{X_i}\dot\gamma, \hat\nabla^\varepsilon_{\dot\gamma} X_i\rangle_\varepsilon,\label{eq:firstpieceofi}
\end{align}
where we used the fact that $\|X_i\|_\varepsilon=1$, that $\hat{T}^\varepsilon$ is completely skew-symmetric, and that by definition $X_i \in \mathfrak{L}_J(\dot\gamma)$ which is means that $T(\dot\gamma,X_i)=0$.

Using the homogeneity properties of Lemma \ref{l:homo}, the fact that $X_i \in \mathfrak{L}_J(\dot\gamma)$, and that $\|\dot\gamma_\ver\| \to 0$ as $\varepsilon\to 0$, we obtain

\begin{equation}\label{eq:approx1}
\hat\nabla^\varepsilon_{\dot\gamma} X_i = -\frac{1}{\|\dot\gamma_\hor\|^2} \sum_{\alpha=1}^m \langle X_i, 
(\nabla_{\dot\gamma} J)_{Z_\alpha}\dot\gamma\rangle J_{Z_\alpha} \dot\gamma +O(\varepsilon).
\end{equation}
Furthermore, again by Lemma \ref{l:homo},
\begin{equation}\label{eq:approx2}
J^\varepsilon_{\dot\gamma} X_i  = \frac{1}{\varepsilon}J_{\dot\gamma_\ver} X_i + O(\varepsilon), \qquad J^\varepsilon_{X_i} \dot\gamma  = O(\varepsilon).
\end{equation}

Replacing \eqref{eq:approx1}-\eqref{eq:approx2} in \eqref{eq:firstpieceofi} and using assumption (ii), we obtain
\begin{equation}
\langle \hat\nabla^\varepsilon_{\dot\gamma} \nabla^\varepsilon_{\dot\gamma} W_i,W_i\rangle  = f\ddot{f} - \frac{f^2}{\|\dot\gamma_\hor\|^2} \left( \sum_{\alpha=1}^m \langle X_i, (\nabla_{\dot\gamma} J)_{Z_\alpha}\dot\gamma \rangle^2 + O(\varepsilon)\right).
\end{equation}
This gives the first term of \eqref{eq:integrand}. For the second term, a lengthy computation specialized to the case $W_i \in \mathfrak{L}_J(\dot\gamma)$, using again (ii), and that $\|\dot\gamma_\ver\|\to 0$ as $\varepsilon \to 0$, yields
\begin{equation}
\hat{R}^{\varepsilon}(X_i,\dot\gamma,\dot\gamma,X_i)  = R^\nabla(X_i,\dot\gamma_\hor,\dot\gamma_\hor,X_i) - \frac{1}{\varepsilon}\langle (\nabla_{X_i}J)_{\dot\gamma_\ver} X_i,\dot\gamma_\hor\rangle +  O(\varepsilon).
\end{equation}
Therefore, for \eqref{eq:integrand} we obtain
\begin{multline}\label{eq:Jacobipiece-1}
i(W_i,W_i)
 = f\ddot{f} + f^2\left[ R^\nabla(X_i,\dot\gamma_\hor,\dot\gamma_\hor,X_i) - \frac{1}{\|\dot\gamma_\hor\|^2}\sum_{\alpha=1}^m \langle X_i, (\nabla_{\dot\gamma_\hor}J)_{Z_\alpha}\dot\gamma_\hor\rangle^2 \right.   \\
 \left. -\frac{1}{\varepsilon}\langle (\nabla_{X_i}J)_{\dot\gamma_\ver}X_i,\dot\gamma_\hor\rangle  +O(\varepsilon)\right].
\end{multline}
Recall that $X_1,\dots,X_{n-m-1}$ constitutes a $g$-orthogonal frame for $\mathfrak{L}_J(\dot\gamma)$. Using assumption (iii), the term of order $\tfrac{1}{\varepsilon}$ in \eqref{eq:Jacobipiece-1} vanishes after taking the sum over $i$. Hence, using the curvature assumption \eqref{eq:curvassII} we obtain
\begin{equation}\label{eq:Jacobipiece}
\sum_{i=1}^{n-m-1} i(W_i,W_i)
 \geq  (n-m-1) f\ddot{f}  + f^2\|\dot\gamma_\hor\|^2 \underbrace{\left(\rho + O(\varepsilon)\right)}_{=:\rho_\ve}.
\end{equation}
Plugging \eqref{eq:Jacobipiece} and \eqref{eq:firstpiece} into \eqref{eq:pretrace} we obtain
\begin{align*}
\frac{\mathrm{tr}_{\mathfrak{L}_J(\dot\gamma)}(\mathrm{Hess}^{\hat\nabla^\varepsilon} r_\ve)}{n-m-1} \leq \dot{f}(r_\varepsilon) - \int_0^{r_\varepsilon} f \left(\ddot{f} +\rho_\varepsilon \|\dot\gamma_\hor\|^2 f\right) dt,
\end{align*}
where $\rho_\varepsilon \to \rho$ uniformly as $\varepsilon \to 0$, and the l.h.s.\ is computed at $y = \gamma(r_\varepsilon)$.

Notice that if $\rho >0$ then also $\rho_\varepsilon >0$ for sufficiently small $\varepsilon$. Thus, we can choose $f$ such that the integrand is zero, satisfying the prescribed boundary conditions $f(0)=f(r_\varepsilon)-1 =0$, that is
\begin{equation}
f(t) = \frac{\sin(\sqrt{\rho_\varepsilon}\|\dot\gamma_\hor\| t)}{\sin(\sqrt{\rho_\varepsilon}\|\dot\gamma_\hor\| r_\varepsilon)}.
\end{equation}
We finally obtain
\begin{equation}\label{eq:maininequalityII}
\frac{\mathrm{tr}_{\mathfrak{L}_J(\dot\gamma)}(\mathrm{Hess}^{\hat\nabla^\varepsilon} r_\ve)}{n-m-1} \leq \sqrt{\rho_\varepsilon} \|\dot\gamma_\hor\|\cot(\sqrt{\rho_\varepsilon}\|\dot\gamma_\hor\| r_\varepsilon).
\end{equation}
A standard argument by contradiction, based on the properties of the Riemannian cut locus, yields for $r_\varepsilon = d_\varepsilon(x,y)$ the following inequality:
\begin{equation}\label{eq:distanceineqII}
d_\varepsilon(x,y) \leq \frac{\pi}{\|\dot\gamma_{\hor}\|\sqrt{\rho_\varepsilon}} = \frac{\pi}{\|\nabla_\hor r_\varepsilon(y)\|\sqrt{\rho_\varepsilon}}.
\end{equation}

We are ready to conclude. By assumption $y\notin \Cut_0(x)$. By Proposition \ref{p:cutapprox}, this means that $y\notin \Cut_\ve(x)$ for all sufficiently small $\ve >0$. Recall that by Remark \ref{rmk:convergence} it holds $\|\nabla_\hor r_\varepsilon\|  \to 1$, and furthermore $d_\ve \to d_0$ as $\ve \to 0$ by Lemma  \ref{l:convergence-of-distance}. We have then
\begin{equation}
d_0(x,y) \leq \frac{\pi}{\sqrt{\rho}}, \qquad \forall y \notin \Cut_0(x).
\end{equation}
Since $\Cut_0(x)$ is a nowhere dense set for all $x\in \M$, it follows that the sub-Riemannian diameter is not greater that $\pi/\sqrt{\rho}$ and since $d_0$ is complete then $\M$ is compact.

We can apply the same argument to the universal cover $\widetilde{\M}$, obtaining that the fundamental group of $\M$ is finite.  \hfill $\qed$

\appendix

\section{Smooth convergence of the canonical variation}\label{b:app}

In this appendix we prove Proposition \ref{p:cutapprox}, concerning the convergence of $d_\ve \to d_0$ and all its derivatives in the compact-open topology outside of the sub-Riemannian cut locus.

We first state two preliminary results. We refer in their proofs to \cite{ABB,Rifford} for basic facts in sub-Riemannian geometry, such as endpoint map, exponential map, singular trajectories and Lagrange multipliers. Only in this section, the notation $\langle\lambda,X\rangle$ denotes the action of one-forms $\lambda \in T^*M$ on vectors $X\in TM$.

\begin{lemma}\label{l:convergence-of-distance}
We have that $d_\ve\to d_0$ uniformly on compact sets of $\M$, as $\ve \to 0$. 
\end{lemma}
\begin{proof}
Firstly, note that $d_\varepsilon$ is increasing as $\varepsilon \to 0$. Let $x,y \in \M$. Fix a sequence $\gamma_{\varepsilon} :[0,1] \to \M$ of minimizing $g_\varepsilon$-geodesics joining $x$ with $y$. Letting $d=d_1$, we have
\begin{equation}
d(\gamma_\varepsilon(t),\gamma_\varepsilon(s)) \leq d_\varepsilon(\gamma_\varepsilon(t),\gamma_\varepsilon(s)) = |t-s| d_{\varepsilon}(x,y) \leq |t-s| d_0(x,y), \qquad \forall \, 0 < \ve \leq 1.
\end{equation}
Letting $t=0$ in the above equation, we have $d(x,\gamma_\varepsilon(s)) \leq d_0(x,y)$, for all $0 < \ve \leq 1$. In particular, the sequence $\gamma_\varepsilon$ is equicontinous and uniformly bounded. By Ascoli-Arzel\`a theorem, and up to extraction of a subsequence, $\gamma_\varepsilon$ converges uniformly to a curve $\gamma_0$.

Let $X_1,\dots,X_n,Z_1,\dots,Z_m$ be a local frame, orthonormal for $g$, defined in a compact neighborhood $K$ of $\gamma_0$. All geodesics $\gamma_\varepsilon$ are contained in $K$, for sufficiently small $\varepsilon$ in the mentioned subsequence. For any control $u \in L^2([0,1],\R^{n+m})$, consider the following ODE on $\M$:
\begin{equation}\label{eq:ODE}
\dot\gamma(t) = \sum_{i=1}^n u_i(t) X_i(\gamma(t)) + \sum_{\alpha=1}^m u_{n+\alpha}(t) Z_\alpha(\gamma(t)), \qquad \gamma(0) = x.
\end{equation}
Let $\mathcal{U}\subset L^2([0,1],\R^{n+m})$ be the subspace such that the solution of \eqref{eq:ODE} is well defined up to time $1$. We consider the endpoint map with initial point $x$, $\End_x: \mathcal{U} \to \M$, which associates with a control $u \in L^2([0,1],\R^{n+m})$ the final point $\gamma(1)$ of the solution to \eqref{eq:ODE}. We stress that $\End_x$ does not depend on $\varepsilon$. For a given control $u \in \mathcal{U}$, associated with a trajectory $\gamma_u$, its energy is
\begin{equation}
J_\varepsilon(u) = \frac{1}{2}\int_0^1\|\dot{\gamma}_u(t)\|^2_{g_\ve} dt, \qquad \ve >0.
\end{equation}
If $\gamma_u$ is a $g_\ve$-geodesic, we have $d_\varepsilon(\gamma_u(0),\gamma_u(1))^2 = 2 J_\varepsilon(u)$, as a consequence of the Cauchy-Schwarz inequality.

To any minimizing $g_\ve$-geodesic $\gamma_\varepsilon$ in the aforementioned sequence, we associate through \eqref{eq:ODE} the corresponding control $u_\varepsilon = (u_{\hor,\varepsilon},u_{\ver,\varepsilon}) \in L^2([0,1],\R^{n+m})$, where
\begin{equation}
u_{\hor,\varepsilon} \in L^2([0,1],\R^n),\qquad u_{\ver,\varepsilon} \in L^2([0,1],\R^m).
\end{equation}
Notice that
\begin{equation}
2J_\varepsilon(u_\varepsilon) = \|u_{\hor,\varepsilon}\|^2_{L^2} + \frac{1}{\varepsilon}\|u_{\ver,\varepsilon}\|^2_{L^2} = d_\varepsilon^2(x,y) \leq d_0^2(x,y).
\end{equation}
It follows that $\|u_{\hor,\varepsilon}\|_{L^2}$ is bounded, and $\|u_{\ver,\varepsilon}\|_{L^2} \to 0$. By weak compactness of bounded sets of $L^2$, we have that, up to extraction of subsequences,
\begin{equation}
u_\varepsilon  \rightharpoonup \bar{u} = (\bar{u}_\hor,0) \in L^2([0,1],\R^{n+m}).
\end{equation}
Notice that the associated trajectory $\gamma_\varepsilon$ converges uniformly to the limit trajectory $\bar{\gamma}$ with control $\bar{u}$, and hence $\End_x(\bar{u})=y$. By the weak semi-continuity of the norm
\begin{equation}\label{eq:firstineq}
\|\bar{u}\|_{L^2} \leq \liminf_{\varepsilon\to 0} \|u_\varepsilon\|_{L^2} \leq  \liminf_{\varepsilon\to 0} d_\varepsilon(x,y) \leq d_0(x,y) \leq \|\bar{u}\|_{L^2}.
\end{equation}
Thus, all inequalities in \eqref{eq:firstineq} are equalities, and thanks to monotonicity, we have
\begin{equation}
d_0(x,y) = \lim_{\varepsilon\to 0} d_\varepsilon(x,y) = \lim_{\varepsilon \to 0} \|u_\varepsilon\|_{L^2} = \|\bar{u}\|_{L^2}.
\end{equation}
In particular $d_\varepsilon \to d_0$, uniformly on compact sets by Dini's theorem. We remark that the weak convergence  and norm convergence of $u_\ve$ imply the $L^2$ convergence $u_\varepsilon \to \bar{u}$.
\end{proof}

For $x\in \M$, and $\ve>0$ we denote by $\exp_x^\ve : T_x^*\M\to \M$ the Riemannian exponential map (on the cotangent space), while $\exp_x = \exp_x^0 : T_x^*\M \to \M$ is the sub-Riemannian one. For all $\ve \geq 0$, we recall that $\exp_x^\ve$ is the projection $\pi : T^*\M \to \M$ of the Hamiltonian flow $e^{t\vec{H}_\ve} : T^*\M \to T^*\M$ associated to the (sub-)Riemannian structure $g_\ve$.

\begin{lemma}\label{l:convergence-of-geodesics}
Let $\gamma_\ve :[0,1]\to \M$ be a sequence of minimizing $g_\ve$-geodesics, starting from $x\in \M$, all contained in a common compact subset of $\M$. Let $\lambda_\ve \in T_x^*\M$ be the initial covector of $\gamma_\ve$, that is $\gamma_\ve(t) = \exp_x^\ve(t \lambda_\ve)$ for all $t\in [0,1]$. Let $y = \lim_{\ve \to 0} \gamma_\ve(1)$, and assume that $y\notin \Cut_0(x)$. Let $\bar\gamma :[0,1] \to \M$ be the unique sub-Riemannian minimizing geodesic between $x$ and $y$, with initial covector $\bar\lambda \in T_x^*\M$, such that $\bar\gamma(t) = \exp_x(t\bar\lambda)$ for all $t\in [0,1]$. Then $\gamma_\ve \to \bar\gamma$ uniformly and $\lambda_\ve \to \bar\lambda$ as $\ve \to 0$.
\end{lemma}
\begin{proof}
Arguing as in the proof of Lemma \ref{l:convergence-of-distance} we have that $\gamma_\ve$ converges up to extraction to a limit constant-speed minimizing sub-Riemannian geodesic $\bar\gamma$ joining $x$ with $y$. Since we are assuming that $y\notin \Cut_0(x)$, there is only one such $\bar{\gamma}$ and thus $\gamma_\ve \to \bar\gamma$ with no need for extraction (since any sub-sequence converges up to extraction to the same limit). 
For sufficiently small $\ve$ the family $\gamma_\ve$ is contained in a compact neighbourhood  of $\bar\gamma$, where a local $g$-orthonormal frame $X_1,\dots,X_n,Z_1\dots,Z_m$ has been fixed, so that the endpoint map $\End_x$ is well defined.  Let also $u_\ve \in L^2([0,1],\R^{n+m})$ be the control of $\gamma_\ve$, for $\ve > 0$, and let $\bar{u}$ be the control of $\bar\gamma$. Again as in the proof of Lemma \ref{l:convergence-of-distance}, $u_\ve \to \bar{u}$ in $L^2$.

Since each $\gamma_\varepsilon$ is a length-minimizer between $x$ and $y_\ve$, for all $\varepsilon>0$ there exists a Lagrange multiplier $\eta_\varepsilon \in T_{y_\ve}^*M$ such that
\begin{equation}\label{eq:lagrange}
\langle \eta_\varepsilon,  D_{u_\varepsilon}\End_x(v) \rangle   = D_{u_\varepsilon} J_\varepsilon(v), \qquad \forall v \in L^2([0,1],\R^{n+m}).
\end{equation}
We now prove that, if $\bar{\gamma}$ is not a singular geodesic, that is, if $\bar{u}$ is not a critical point of $\End_x$, then $\eta_\varepsilon$ converges to the Lagrange multiplier $\bar{\eta}$ of $\bar{\gamma}$. We remark that if $\bar{\gamma}$ is non-singular, then the Lagrange multiplier $\bar{\eta}$ will be unique.

In order to simplify the r.h.s.\ of \eqref{eq:lagrange}, we restrict ourselves to sub-Riemannian variations, that is to $v \in L^2([0,1],\R^{n}) \subset L^2([0,1],\R^{n+m})$. In this case $D_{u_\varepsilon} J_\varepsilon(v) = D_{u_\varepsilon} J(v)$, where $J=J_1$, and \eqref{eq:lagrange} reads
\begin{equation}\label{eq:lagrange2}
\langle \eta_\varepsilon,  D_{u_\varepsilon}\End_x(v) \rangle   = (u_\varepsilon, v)_{L^2}, \qquad \forall v \in L^2([0,1],\R^{n}).
\end{equation}
We claim that for any $u$ sufficiently close to $\bar{u}$ and smooth vector field $V$ on $\M$, there exists a sub-Riemannian variation $v(u) \in L^2([0,1],\R^{n})$ (depending on the choice of $V$) such that $D_u \End_x (v(u)) = V_z$, where $z=\End_x(u)$, and the map $u \mapsto v(u)$ is continuous. To prove the claim, note that thanks to our assumption, $\bar{u}$ is a regular point for the sub-Riemannian endpoint map (that is the endpoint map restricted to controls of the form $u =(u_\hor,0)$). In other words, there exist $v_1,\dots,v_{n+m} \in L^2([0,1],\R^{n})$ such that $D_{\bar{u}} \End_x(v_i) \in T_y \M$, for $i=1,\dots,n+m$, are independent vectors. Since $u\mapsto D_u\End_x$ is continuous, the same holds (with the same $v_i$) for all controls $u$ in an open neighbourhood $U \subset L^2([0,1],\R^{n+m})$ of $\bar{u}$. Up to restricting $U$, we can assume that $\End(U) \subset O$, where $O$ is a coordinated neighbourhood of $\End_x(\bar{u})=y$. Let
\begin{equation}
\Phi : U \times \R^{n+m}  \to U \times O, \qquad \Phi(u,\alpha) = \left(u,\End_x\left(u+\sum_{i=1}^{n+m}\alpha_i v_i\right)\right).
\end{equation}
By the inverse function theorem, there exists a local inverse. More precisely, and up to restriction of the domains, there exists an inverse $\Phi^{-1} : U \times O \to U \times \R^{n+m}$. Let $\delta : (-a,a) \to \M$ be a smooth curve with $\delta(0) = z$ and $\dot\delta(0) =V_z \in T_z \M$. We have
\begin{equation}
\Phi^{-1}(u,\delta(\tau)) = (u,\alpha(u,\tau)),\quad\text{with}\quad \End_x(u+\sum\alpha_i(u,\tau)v_i) = \delta(\tau),
\end{equation}
and $\alpha : U \times (-a,a) \to \R^{n+m}$ is smooth. The map $v : U \to L^2([0,1],\R^{n})$ given by
\begin{equation}
u\mapsto \sum_{i=1}^{n+m} \frac{\partial \alpha_i}{\partial \tau}(u,0) v_i,
\end{equation}
is smooth, and satisfies $D_u \End_x(v(u)) = V_z$, by construction. This proves the claim.

By plugging such a choice of $v(u_\ve)$ in \eqref{eq:lagrange2}, we obtain
\begin{equation}\label{eq:vectorfield}
\langle \eta_\varepsilon, V\rangle = (u_\varepsilon, v(u_\varepsilon))_{L^2}.
\end{equation}
In particular, since $u_\varepsilon \to \bar{u}$ and $v(u_\ve) \to v(\bar{u})$, we have proved that for any choice of smooth vector field $V$ on $\M$, the l.h.s.\ of \eqref{eq:vectorfield} is convergent. This is equivalent to the convergence of $\eta_\varepsilon$ to some $\bar{\eta} \in T_y^*\M$. Taking the limit in \eqref{eq:lagrange} we obtain
\begin{equation}
\langle \bar\eta,  D_{\bar{u}}\End_x(v) \rangle   = (\bar{u},v)_{L^2}, \qquad \forall v \in L^2([0,1],\R^{n}).
\end{equation}
By definition, this means that $\bar{\eta}$ is the unique normal Lagrange multiplier associated with the minimizing sub-Riemannian geodesic $\bar{\gamma}$ with control $\bar{u}$.

It is well-known that the normal extremal lift $\lambda_\ve :[0,1] \to T^*\M$ of $\gamma_\ve$ (resp. the normal extremal lift $\bar\lambda$ of $\bar\gamma$) can be recovered from the Lagrange multiplier by the Hamiltonian flows $e^{t\vec{H}_\ve}:T^*\M \to T^*\M$ given by the Hamiltonian $H_\ve$ associated with $g_\ve$ for all $\ve~>~0$ (resp.\ the sub-Riemannian Hamiltonian flow $e^{t\vec{H}}$ associated with the sub-Riemannian Hamiltonian $H$).  More precisely,
\begin{equation}
\lambda_\ve(t) = e^{(t-1)\vec{H}_\ve}(\eta_\ve), \qquad \text{and} \qquad \bar{\lambda}(t) = e^{(t-1)\vec{H}}(\bar{\eta}).
\end{equation}
Since $H_\ve \to H$ uniformly on compact sets as functions on $T^*\M$, it follows that $\lambda_\ve(t) \to \bar{\lambda}(t)$ uniformly for $t\in [0,1]$, which is thus equivalent to the convergence of the initial covector $\lambda_\ve(0) \to \bar{\lambda}(0)$.
\end{proof}

%

We are now ready for the proof of Proposition \ref{p:cutapprox}.

\begin{proof}[Proof of Proposition \ref{p:cutapprox}]
Fix $x\in \M$, and let $F: T_{x}^*\M \times (-1,1) \to \M \times (-1,1)$ be the smooth map given by
\begin{equation}
F(\lambda,\varepsilon) = \left(\exp^\ve_x(\lambda),\ve\right)
\end{equation}
where $\exp^\ve_x :T_x^*\M \to \M$ denotes the (sub-)Riemannian exponential map associated with the structure $g_\ve$. Notice that the latter is well defined and smooth for all values of $\ve \in (-1,1)$, as the projection $\pi :T \M \to \M$ of the Hamiltonian flow of $H_\ve$. By our assumption, there exists a unique minimizing $g_0$-geodesic joining $x$ with $y$, with initial covector $\bar{\lambda}$, and the latter is a regular point for $\exp_x^0$. The differential of $F$ at $(\bar\lambda,0)$ is given by
\begin{equation}
(d_{(\bar{\lambda},0)}F) (\dot\lambda,\dot\ve) = \left((d_{\bar\lambda}\exp^\ve_x) \dot\lambda + \frac{\partial F(\lambda,\ve)}{\partial \ve} \dot\ve,\dot\ve\right), \qquad (\dot\lambda ,\dot\ve)\in T_{(\bar\lambda,0)} U \times \R.
\end{equation}
Since $\rank (d_{\bar\lambda}\exp^\ve_x)= n$, by the inverse function theorem there exist neighbourhoods $U\subset T_x^*\M$ of $\bar{\lambda}$, $I=(-\ve',\ve' ) \subset (-1,1)$ of $0$, and $V\subset \M$ of $y$ and a smooth diffeomorphism $G : V\times I \to U\times I$ such that $F \circ G = \mathrm{Id}_{V\times I}$.

This means that for all $|\ve| < \ve'$ and for all $z \in V$ there exists a unique $\lambda_{\ve,z} \in U$ such that $\exp_x^\ve(\lambda_{\ve,z}) = z$, corresponding to a geodesic $\gamma_{\ve,z}$ between $x$ and $z$. Clearly $x$ and $z$ are not conjugate along this geodesic. We claim that, up to restricting further $V$ and $\ve'$, the curve $\gamma_{\ve,z}$ is the unique minimizing $g_\ve$-geodesic joining $x$ with $z$.

Indeed, assume by contradiction that there is a sequence $z_n \to y$ and $\ve_n \to 0$ and minimizing geodesics $\gamma_n$ joining $x$ with $z_n$, with initial covector $\lambda_n \notin U$, and whose length is not greater than the length of $\gamma_{\ve_n,z_n}$. Let also $\bar\gamma$ be the minimizing $g_0$-geodesic between $x$ and $y$. By Lemma \ref{l:convergence-of-geodesics}, we have that $\gamma_n \to \bar\gamma$, and $\lambda_{\ve_n} \to \bar\lambda$, which is a contradiction, as $\lambda_{\ve_n}$ is separated from $\bar\lambda$.

To prove the smoothness part of the statement, we remark that the map $(z,\ve) \mapsto \lambda_{\ve,z}$ built above is smooth. Furthermore by minimality one has $d_\ve(x,z)^2 = 2H_\ve(\lambda_{z,\ve})$. That is, in terms of a $g$-orthonormal frame for $g$ in a neighbourhood of $V$, it holds
\begin{equation}
d_\ve(x,z)^2 =\sum_{i=1}^n \langle \lambda_{\ve,z},X_i(z)\rangle^2 + \ve \sum_{\alpha=1}^m \langle \lambda_{\ve,z},Z_\alpha(z)\rangle^2,
\end{equation}
thus proving that the squared distance $d_\ve(x,z)^2$ is jointly smooth in both variables for $|\ve|< \ve'$ and $z \in V$. Since $V$ is separated from $x$, we have $d_\ve(x,z)>0$ for all $z\in V$, and the joint smoothness of $d_\ve(x,z)$ follows.
\end{proof}

\section{On the index lemma and adjoint connections}\label{a:app}

Let $(\M,g)$ be a Riemannian manifold. Let $\nabla$ be a metric connection, with torsion $T$. We also use the notation $g = \langle\;,\,\rangle$ for the scalar product. The endomorphism $J : T\M \to T\M$ is defined by $\langle J_X Y ,Z \rangle = \langle X, T(Y,Z)\rangle$.

\begin{lemma}
Let $\gamma :[0,r] \to \M$ be a smooth length-parametrized curve. Let $\Gamma(t,s):[0,r] \times [-\varepsilon,\varepsilon] \to \M$ be a one-parameter variation of $\gamma$ with fixed endpoints. Let $X = \partial_s \Gamma|_{s=0}$. Then
\begin{equation}
\left.\frac{d}{ds}\right|_{s=0} \ell(\Gamma(\, \cdot \, ,s)) = -\int_0^r \langle \nabla_{\dot\gamma}\dot\gamma + J_{\dot\gamma} \dot\gamma,X\rangle dt.
\end{equation}
In particular, $\gamma$ is a locally distance minimizing geodesic if and only if $\nabla_{\dot\gamma}\dot\gamma + J_{\dot\gamma}\dot\gamma = 0$.
\end{lemma}
Given an affine connection $D$, with torsion $\Tor$, the adjoint connection $\hat{D}$ is defined by the formula $\hat{D}_X Y = D_X Y - \Tor(X,Y)$. Observe that $\hat{\hat{D}} = D$.
\begin{proof}
Consider the vector fields $S= \Gamma_* \partial_s$ and $\Gamma_*\partial_t$, which in particular satisfy $T(t,0) = \dot\gamma(t)$ and $S(t,0) = X(t)$. By definition, if $U \subseteq [0,r] \times [-\varepsilon,\varepsilon]$ any sufficiently small open and $\bar{S}$ and $\bar{T}$ are arbitrary vector fields on $\M$ extending $S|_U$ and $T|U$, then $[\bar{S}, \bar{T}]|_{\Gamma(U)} =0$. We hence have the following identities,
\begin{align}
\label{STswitch} \nabla_S T & = \hat \nabla_T S \\
\label{STTswitch} \nabla_S \nabla_T T & = R(S,T) T + \nabla_T \nabla_S T
\end{align}
Using the identity \eqref{STswitch}, we obtain
\begin{align*}
\left.\frac{d}{ds}\right|_{s=0} \ell(\Gamma_s) & =	 \int_0^r \frac{\langle \nabla_S T, T\rangle}{\|T\|} dt = \int_0^r \langle \nabla_T S  + T(S,T), T\rangle dt \\
& = \langle S, T \rangle|_{t=0}^{t=r} - \int_0^r \langle \nabla_T T+ J_T T, S\rangle  dt  = -\int_0^r \langle \nabla_{\dot\gamma}\dot\gamma + J_{\dot\gamma} \dot\gamma,X\rangle dt,
\end{align*}
where the integrand is computed at $s=0$, and $\|T(t,0)\| = \|\dot\gamma(t)\| = 1$.
\end{proof}
\begin{lemma}\label{l:construction}
Let $\nabla$ be a metric connection. Then there always exists an associated connection $D$ which is metric with metric adjoint $\hat D$, given by
\begin{equation}
D_XY := \nabla_X Y + J_X Y .
\end{equation}
\end{lemma}
\begin{proof}
Since $J_X : T\M \to T\M$ is skew-symmetric, $D$ is metric. Let $\Tor$ be the torsion of $D$. We have:
\begin{align*}
\hat{D}_X Y & = D_X Y  - \Tor(X,Y) = \nabla_X Y - T(X,Y) - J_X Y + J_Y X \\
& = \nabla_X Y -T(X,Y) + J_Y X.
\end{align*}
Since $B_X Y := -T(X,Y) + J_Y X$ satisfies $\langle B_X Y,Y \rangle = 0$, $\hat{D}$ is metric.
\end{proof}
We also highlight the following observation which is simple to verify.
\begin{lemma}\label{l:constructionskew}
The adjoint $\hat{D}$ of a metric connection $D$ is metric if and only if the tensor $X,Y,Z\mapsto \langle \Tor(X,Y),Z\rangle$ is completely skew-symmetric.
\end{lemma}

From now on, we assume that $D$ is a metric connection such that $\hat{D}$ is metric, which is always possible thanks to  Lemma~\ref{l:construction}. In particular, the geodesic equation  is $D_{\dot\gamma}\dot\gamma = 0$ or, equivalently, $\hat{D}_{\dot\gamma}\dot\gamma = 0$.

\begin{lemma}
Let $\gamma :[0,r] \to \M$ be a geodesic. A vector field $V$ along $\gamma$ is a Jacobi field if and only if
\begin{equation}
D_{\dot\gamma}\hat{D}_{\dot\gamma} V + R(V,\dot\gamma)\dot\gamma = 0.
\end{equation}
\begin{proof}
Let $x = \gamma(0)$, and $\Gamma :[0,r] \times (-\varepsilon,\varepsilon) \to \M$ be a one-parameter variation associated with $V$. In particular $\Gamma(t,0) = \gamma(t)$, $\Gamma(t,s)$ is a geodesic for all fixed $s \in (-\varepsilon,\varepsilon)$, and $\partial_s \Gamma(t,0) = V(t)$. Let $T = \Gamma_* \partial_t$ and $S = \Gamma_* \partial_s$. Using the identity $D_T T = 0$ and \eqref{STTswitch}, we deduce
\begin{equation}
0 = D_S D_T T  = R(S,T)T + D_T D_S T = R(S,T)T + D_T \hat{D}_T S.
\end{equation}
Computing the above at $s=0$ yields the statement.
\end{proof}
\end{lemma}

\begin{lemma}
Let $\gamma :[0,r] \to \M$ be a unit-speed geodesic joining $x \in \M$ with $y \notin \mathbf{Cut}(x)$. Let $r = d(x,\cdot)$, and $X \in T_{\gamma(r)} \M$, with $X \perp \dot\gamma(r)$. Then
\begin{equation}
\mathrm{Hess}^D(r)(X,X) = \int_0^r \left(\langle D_{\dot\gamma} V, \hat{D}_{\dot\gamma} V\rangle - \langle R(V,\dot\gamma) \dot\gamma,V \rangle \right)dt,
\end{equation}
where, $\mathrm{Hess}^D$ denotes the Hessian with respect to the connection $D$ and, in the integrand, $V$ denotes the Jacobi field along $\gamma$ such that $V(0) = 0$ and $V(r) = X$.
\end{lemma}
We remark that using the Jacobi equation and the fact that $D$ is metric, one has
\begin{equation}
\mathrm{Hess}^D(r)(X,X) =\langle X,D_{\dot{\gamma}} V(r)\rangle.
\end{equation}

\begin{proof}
Consider a geodesic curve $\sigma : (-\varepsilon,\varepsilon) \to \M$ with $\sigma(0) = \gamma(r)$, and with $\dot\sigma(0) = X$. For any $s \in (-\varepsilon,\varepsilon)$ let $t\mapsto \Gamma(t,s)$ be the unique length-parametrized geodesic joining $\gamma(0)$ with $\sigma(s)$. This family is well defined and smooth since $\gamma(r)$ is outside of the cut locus, provided that $\varepsilon$ is small. Let $S= \Gamma_* \partial_s$ and $T = \Gamma_*\partial_t $. Indeed, $T(t,0) = \dot\gamma(t)$ and $S(t,0) = V(t)$. Notice that by construction $\|T(t,0)\|=1$. Furthermore, the boundary conditions and Jacobi equation imply that $\langle V, \dot\gamma\rangle = 0$ for all times. Then, with the understanding that everything is computed at $s=0$, we have
\begin{align*}
D^2 r (X,X) & = \frac{d^2}{ds^2} \int_0^r \| T\| dt = \frac{d}{ds} \int_0^r \frac{1}{\|T\|} \langle D_ST,T\rangle dt \\
& = \frac{d}{ds} \int_0^r \frac{1}{\|T\|} \langle D_TS + \Tor(S,T),T\rangle dt \\
& = \int_0^r \left( -\langle D_ S T,T\rangle \langle D_TS,T\rangle + \langle D_S D_T S,T\rangle + \langle D_T S, D_S T\rangle\right) dt \\
& = \int_0^r \langle D_S D_T S,T\rangle + \langle D_T S, \hat{D}_T S\rangle\\
& = \int_0^r \left(\langle D_T S, \hat{D}_T S\rangle + R(S,T,S,T) + \langle D_T D_S S,T\rangle \right) dt\\
& = \langle D_{\dot\sigma}\dot\sigma(0), \dot\gamma(r)\rangle + \int_0^r \left(  \langle D_T S, \hat{D}_T S\rangle - R(S,T,T,S) \right) dt\\
& = \int_0^r \left(\langle D_{\dot\gamma} V, \hat{D}_{\dot\gamma} V\rangle - R(V,\dot\gamma,\dot\gamma,V)\right)dt,
\end{align*}
where, in the second-to-last line, we used the fact that $\sigma$ is a geodesic.
\end{proof}

\begin{definition}
Let $\gamma :[0,r] \to \M$ be a unit-speed geodesic. For any vector field $W$ along $\gamma$, we define the index form as
\begin{equation}
I(W,W) = \int_0^r \left(\langle D_{\dot\gamma} W, \hat{D}_{\dot\gamma} W\rangle - \langle R(W,\dot\gamma)\dot\gamma,W \rangle \right)dt.
\end{equation}
\end{definition}
We emphasize the fact that $I(W,W)$ coincides for all compatible connections $D$ with skew-symmetric torsion, including the Levi-Civita connection of $g$. We recall the classical index, Lemma see e.g.~\cite[Cor.\ 6.2]{Lee97} and \cite[Thm.\ 1.1.11]{Wan14}

\begin{lemma}\label{Lemma index}
Let $\gamma :[0,r] \to \M$ be a unit-speed geodesic free of conjugate points. Let $W$ be a vector field along $\gamma$, and let $V$ be the Jacobi field along $\gamma$ such that $V(0) = W(0)$ and $V(r) = W(r)$. Then
\begin{equation}
I(V,V) \leq I(W,W),
\end{equation}
with equality if and only if $V=W$.
\end{lemma}

We summarize all the above results in a single theorem.
\begin{theorem}\label{t:index}
Let $(\M,g)$ be a Riemannian manifold, with metric connection $\nabla$. There always exists a metric connection $D$ which also has a metric adjoint, given by
\begin{equation}
D_X Y = \nabla_X Y + J_X Y.
\end{equation}
A length-parametrized curve $\gamma :[0,r] \to \M$ is a geodesic if and only if $D_{\dot\gamma}\dot\gamma = 0$ or, equivalently, if $\hat{D}_{\dot\gamma}\dot\gamma = 0$, and a vector field $V$ along $\gamma$ is a Jacobi field if and only if
\begin{equation}
D_{\dot\gamma} \hat{D}_{\dot\gamma} V + R(V,\dot\gamma)\dot\gamma =0,
\end{equation}
where $R$ is the curvature associated with $D$. 

The index lemma holds, that is if $\gamma(r)\notin \Cut(\gamma(0))$, for any vector field $W \perp \dot\gamma$ along $\gamma$ such that $V(0) =0$ it holds
\begin{equation}
\mathrm{Hess}^D(r)(V,V) \leq I(W,W),
\end{equation}
with equality if and only if $W$ is a Jacobi field perpendicular to $\dot\gamma$.
\end{theorem}

\bibliographystyle{habbrv}
\bibliography{biblio}

\Addresses

\end{document}